\documentclass[a4paper, 11pt]{amsart}
\usepackage{amsfonts, amsthm, amssymb, amsmath}
\usepackage{mathrsfs,array}
\usepackage{xy}
\usepackage{hyperref}
\usepackage{verbatim}
\usepackage{cancel}
\usepackage[latin1]{inputenc}
\usepackage{tikz-cd}
\usepackage{bbm}
\usepackage{stmaryrd}

\input xy
\xyoption{all}

\newtheorem{theo}{Theorem}[section]
\newtheorem{prop}[theo]{Proposition}
\newtheorem{defi}[theo]{Definition}
\newtheorem{lemm}[theo]{Lemma}
\newtheorem{coro}[theo]{Corollary}
\newtheorem{hope}[theo]{Hope}
\newtheorem{conj}[theo]{Conjecture}
\newtheorem{rema}[theo]{Remark}

\newcommand{\wh}{\widehat}
\newcommand{\wt}{\widetilde}
\newcommand{\mb}{\mathbb}
\newcommand{\mc}{\mathcal}
\newcommand{\mf}{\mathfrak}

\newcommand{\sub}{\subseteq}

\newcommand{\Zp}{\mathbb{Z}_{p}}
\newcommand{\Qp}{\mathbb{Q}_{p}}
\newcommand{\Fp}{\mathbb{F}_{p}}
\newcommand{\Q}{\mathbb{Q}}
\newcommand{\R}{\mathbb{R}}
\newcommand{\C}{\mathbb{C}}
\newcommand{\ol}{\overline}
\newcommand{\G}{\Gamma}

\newcommand{\cF}{\mathcal{F}}
\newcommand{\bu}{\bullet}

\newcommand{\Z}{\mathbb{Z}}

\newcommand{\A}{\mathbb{A}}
\newcommand{\pr}{\prime}
\newcommand{\tH}{\widetilde{H}}
\newcommand{\tG}{\widetilde{G}}
\newcommand{\tX}{\widetilde{X}}
\newcommand{\an}{\mathrm{an}}
\newcommand{\rank}{\mathrm{rank}}

\DeclareMathOperator{\Ext}{Ext}

\DeclareMathOperator{\Hom}{Hom}

\DeclareMathOperator{\Sh}{Sh}

\DeclareMathOperator{\Spa}{Spa}

\DeclareMathOperator{\Ind}{Ind}
\DeclareMathOperator{\Map}{Map}
\DeclareMathOperator{\GL}{GL}
\DeclareMathOperator{\SL}{SL}
\DeclareMathOperator{\GSp}{GSp}
\DeclareMathOperator{\Ker}{Ker}

\DeclareMathOperator{\Spd}{Spd}

\title{Perfectoid Shimura varieties and the Calegari--Emerton conjectures}

\author{David Hansen} 
\address{David Hansen\\Department of Mathematics\\
National University of Singapore, 10 Lower Kent Ridge Road, Singapore 119076}
\email{dhansen@nus.edu.sg}
\urladdr{http://www.davidrenshawhansen.net}

\author{Christian Johansson}
\address{Christian Johansson\\Department of Mathematical Sciences, Chalmers University of Technology and the University of Gothenburg, 412 96 Gothenburg, Sweden}
\email{chrjohv@chalmers.se}
\urladdr{http://www.math.chalmers.se/~chrjohv/}

\date{\today}

\begin{document}

\begin{abstract}
We prove many new cases of a conjecture of Calegari-Emerton describing the qualitative properties of completed cohomology. The heart of our argument is a careful inductive analysis of completed cohomology on the Borel-Serre boundary. As a key input to this induction, we prove a new perfectoidness result for towers of minimally compactified Shimura varieties, generalizing previous work of Scholze.
\end{abstract}

\maketitle

\tableofcontents

\section{Introduction}

\subsection{Motivation for completed cohomology}
This paper is motivated by the notion of reciprocity in the Langlands program. Let $G / \mathbb{Q}$ be a connected reductive group. Roughly speaking, reciprocity is the expectation that there should be some precise relationship between

\begin{itemize}
\item algebraic automorphic representations $\pi$ of $G(\mathbb{A_Q})$, and

\smallskip

\item $p$-adic Galois representations $\rho: \mathrm{Gal}(\ol{\Q} / \Q) \to \, ^L G(\overline{\Q}_p) $ which are geometric in the sense of Fontaine-Mazur. 
\end{itemize}

For a more precise conjectural formulation of this relationship, we refer the reader to \cite{clo, bg}. While there are many partial results, the general problem of reciprocity seems very difficult to attack, for (at least) two reasons:

\begin{enumerate}
\item  Algebraic automorphic representations are inherently of an archimedean/real-analytic nature, while $p$-adic Galois representations are (of course) inherently $p$-adic. 

\smallskip

\item Algebraic automorphic representations are rigid, while $p$-adic Galois representations naturally deform into positive-dimensional families.
\end{enumerate}

These observations suggest that one should try to bridge the gap, by seeking a genuinely $p$-adic variant of the notion of automorphic representation, which is flexible enough to accommodate all $p$-adic Galois representations.  At present, the most satisfactory theory of ``$p$-adic automorphic representations'' is the notion of \emph{completed (co)homology}, introduced by Emerton \cite{em}. 

\medskip

Let us recall the key definitions; we refer the reader to the body of the paper for any unexplained notation. Fix a connected reductive group $G/ \Q$. Let $A \sub G$ be the maximal $\Q$-split central torus, and let $K_\infty \sub G(\R)$ be a maximal compact subgroup. Let $X^G=G(\R)/A(\R)K_\infty$ be the (connected) symmetric space for $G$; we write $X$ for $X^G$ if $G$ is clear. For any open compact subgroup $K \sub G(\A_f)$, we have the associated locally symmetric space $X_K = G(\Q)^+ \backslash (X \times G(\A_f)) / K$.

\begin{defi} Let $K^{p} \sub G(\A_{f}^{p})$ be any open compact subgroup. Then we define completed cohomology for $G$ with tame level $K^p$ as 
\[ \wt{H}^{\ast}(K^p) = \varprojlim_{n} \varinjlim_{K_p \sub G(\Qp)} H^{\ast}(X_{K^p K_p}, \Z/p^n).
\]
Similarly, we define completed homology for $G$ with tame level $K^p$ as  
\[ \wt{H}_{\ast}(K^p) = \varprojlim_{K_p \sub G(\Qp)} H_{\ast}(X_{K^p K_p}, \Zp).
\]
We also define compactly supported completed cohomology $\wt{H}^{\ast}_c(K^p)$ and completed Borel-Moore homology $\wt{H}_{\ast}^{BM}(K^p)$ by the obvious variants on these recipes.
\end{defi}

By construction, these spaces admit commuting actions of $G(\Qp)$ and a ``big'' Hecke algebra $\mathbb{T}(K^p)$, and the $G(\Qp)$-actions are continuous for the natural topologies. Moreover, these spaces are not ``too big''. In particular, they are all $p$-adically separated and complete with bounded $p^\infty$-torsion. Additionally, $\wt{H}_{\ast}$ and $\wt{H}_{\ast}^{BM}$ are finitely generated as modules over the completed group ring $\Zp \llbracket K_p \rrbracket$ for any open compact subgroup $K_p \sub G(\Qp)$, while $\wt{H}^{\ast}(K^p)[\tfrac{1}{p}]$ and $\wt{H}^{\ast}_c(K^p)[\tfrac{1}{p}]$ are naturally admissible unitary $\Qp$-Banach space representations of $G(\Qp)$. 

\medskip

The main motivations for considering completed (co)homology are summarized in the following conjecture, which we don't attempt to formulate precisely. For a more careful discussion, we refer the reader to \cite{ce} and \cite{em-icm}.

\begin{hope} Let $\psi: \mathbb{T}(K^p) \to \ol{\Q}_p$ be a system of Hecke eigenvalues occurring in $\wt{H}^{\ast}(K^p)[\tfrac{1}{p}]$. Then there exists a continuous, odd, almost everywhere unramified Galois representation $\rho_{\psi} : \mathrm{Gal}(\ol{\Q} / \Q) \to \, ^C G(\ol{\Q}_p)$ which matches $\psi$ in the usual sense. Moreover, the $\psi$-isotypic part of $\wt{H}^{\ast}(K^p)[\tfrac{1}{p}]$, as a $\Qp$-Banach space representation of $G(\Qp)$, should (up to multiplicities) depend only on $\rho_{\psi} |_{\mathrm{Gal}(\ol{\Q}_p / \Qp)}$. 

Finally, \emph{every} (suitable) continuous, odd, almost everywhere unramified Galois representation $\rho: \mathrm{Gal}(\ol{\Q} / \Q) \to \, ^C G(\overline{\Q}_p) $ should occur in this way. 
\end{hope}

Here $^C G$ denotes the $C$-group of $G$ as defined in \cite{bg}, which is an extension of $^L G$. When $G=\mathrm{GL}_2 / \mathbb{Q}$, this is (an imprecise version of) a theorem of Emerton \cite{em-lgc}. However, in general, very little is known. As mentioned, the precise formulation of this conjecture should not be taken too seriously. The reader wondering about the appearance of the $C$-group and what ``suitable'' means might want to consider the case $G=\mathrm{PGL}_2 / \Q$.

\subsection{Main results}

In this paper, we study the qualitative properties of completed (co)homology, which are encapsulated in a beautiful conjecture of Calegari--Emerton. To state this conjecture, we need a small amount of additional notation. If $G / \Q$ is a connected reductive group, we define nonnegative integers $l_0 = \rank\,G(\R) - \rank\,A(\R)K_\infty$ and $q_0 = \frac{\dim X^G - l_0}{2}$. Roughly speaking, for semisimple groups $l_0$ measures the failure of $G(\R)$ to admit discrete series representations, while $q_0$ is the lowest degree in which the locally symmetric spaces $X_K$ should have ``interesting'' cohomology. 

\begin{conj}[Calegari--Emerton] \label{ceconj} Let $G /\Q$ be a connected reductive group. Let $q_0$ and $l_0$ be the invariants of $G$ defined above. Let $K^p \sub G(\A_{f}^{p})$ be any open compact subgroup. Then

\begin{enumerate}
\item For all $i > q_0$, $\wt{H}^{i}_{c}(K^p) = \wt{H}^i(K^p)=0$.

\smallskip

\item For all $i > q_0$, $\wt{H}_{i}^{BM}(K^p) = \wt{H}_i(K^p)=0$, and $\wt{H}_{q_0}^{BM}(K^p)$ and $\wt{H}_{q_0}(K^p)$ are $p$-torsion-free.

\smallskip

\item For any compact open pro-$p$ subgroup $K_p \sub G(\Qp)$, the groups $\wt{H}_{i}(K^p)$ and $\wt{H}_{i}^{BM}(K^p)$ have codimension $ \geq q_0 + l_0 - i$ over the completed group ring $\Zp \llbracket K_p \rrbracket$ for any $i < q_0$.

\smallskip

\item The groups $\wt{H}_{q_0}(K^p)$ and $\wt{H}_{q_0}^{BM}(K^p)$ have codimension exactly $l_0$.
\end{enumerate}

\end{conj}

The individual portions of this conjecture are far from independent, and in fact there are natural implications $(1) \Rightarrow (2) \Rightarrow (3)$. Amusingly, these implications are ``asymmetric'' in the sense that (1) for $\tilde{H}^\ast$ implies (2) for $\tilde{H}_{\ast}$ implies (3) for $\tilde{H}_{\ast}^{BM}$, and similarly (1) for $\tilde{H}^{\ast}_c$ implies (2) for $\tilde{H}_{\ast}^{BM}$ implies (3) for $\tilde{H}_{\ast}$.

\medskip

Let us discuss what was previously known about this conjecture.
\begin{itemize}

\item For some groups of small rank (e.g. $\GL_2$, or $\mathrm{Res}_{K/\Q}\GL_2$ for $K/\Q$ quadratic, or $\mathrm{GSp}_4$), one can prove Conjecture \ref{ceconj} by hand using various tricks involving the congruence subgroup property, the cohomological dimension bounds of \cite{bs}, Poincar\'e duality, etc. However, these methods quickly run out of steam. 

\smallskip

\item When $G$ is semisimple and $l_0 = 0$, part (4) of the conjecture was proved by Calegari--Emerton \cite{ce-annals}, as a consequence of Matsushima's formula and limit multiplicity results for discrete series representations. 

\smallskip

\item When $G$ admits a Shimura datum of Hodge type, Scholze proved part (1) of Conjecture \ref{ceconj}, \emph{but for $\wt{H}^{\ast}_c$ only}, by perfectoid methods \cite{scho}. Shen \cite{shen} later treated the case when $G$ admits a \emph{compact} Shimura variety of abelian type and satisfies $l_0(G)=0$.

\smallskip

\item For the unitary Shimura varieties treated in \cite{caraiani-scholze2}, Conjecture \ref{ceconj}(1) for $\wt{H}^\ast$ follows from \cite[Theorem 2.6.2, Lemma 4.6.2]{caraiani-scholze2}. We make some further comments in Remark \ref{comments on cs}.
\end{itemize}

The main result of this paper is the following theorem (cf. Theorems \ref{main semisimple}, \ref{main reductive}, and \ref{general reductive}).

\begin{theo}\label{maintheorem} Let $G / \Q$ be a semisimple group such that $X$ is a Hermitian symmetric space and $(G,X)$ is a connected Shimura datum of pre-abelian type. Then Conjecture \ref{ceconj} is true for $G$.

More generally, let $G / \Q$ be a connected reductive group such that $Z(G)$ satisfies the Leopoldt conjecture and such that $G^{der}$ admits a connected Shimura datum of pre-abelian type.  Then Conjecture \ref{ceconj} is true for $G$.

Moreover, there exists a (computable) $j\leq q_0$ such the natural maps $\wt{H}_c^i \to \wt{H}^i$ and $\wt{H}_i \to \wt{H}_i^{BM}$ are isomorphisms for all $i > j$, and surjective in degree $i=j$.
\end{theo}

The assumptions on $G$ here guarantee that $l_0(G^{der}) = 0$, which allows us to prove part (4) of Conjecture \ref{ceconj} by a fairly straightforward analysis combining the results of \cite{ce-annals} with the Leopoldt conjecture for $Z(G)$. By our previous remarks, the whole conjecture now follows if we can prove part (1).  Note that when $ l_0 = 0$ and $X$ is a Hermitian symmetric domain, part (1) of the conjecture simply asserts that $\wt{H}^{i}_c=\wt{H}^{i}=0$ for all $i> d = \dim_{\C} X$. It is this vanishing conjecture which we focus on.

\medskip

Our proof of the vanishing conjecture builds on Scholze's methods and combines them with some new ideas. Roughly speaking, we first reduce to the case where $(G,X)$ is a connected Shimura datum of pre-abelian type, and then proceed in two steps:

\smallskip

\textbf{Step One.} We prove the vanishing of $\wt{H}^{i}_c$ for $i>d$ by pushing Scholze's methods to their limit.

\smallskip

\textbf{Step Two.} We prove the vanishing of $\wt{H}^{i}$ for $i>d$ by a careful analysis of boundary cohomology, using Step One for $G$ \emph{and} for various auxiliary almost direct factors of Levi subgroups related to the boundary strata of the minimal compactification.

\medskip

Let us now describe these steps in more detail.

\subsection{Step One: $p$-adic methods}

As described above, the proof of Theorem \ref{maintheorem} proceeds in two essentially distinct steps. In the first step, we prove the vanishing of $\wt{H}^{i}_c(K^p)$ for $i$ above the middle degree, using the $p$-adic geometry of Shimura varieties. For Shimura data of Hodge type, this is one of the main results of \cite{scho}, where it is deduced from the existence of perfectoid Shimura varieties of Hodge type (with infinite level at $p$). We thus need to generalize the geometric results of \cite{scho} to a wider class of Shimura data. To this end, we prove the following theorem.

\begin{theo} \label{perfshim} Let $(G,X)$ be a Shimura datum of \emph{pre-abelian type}, with reflex field $E$. Fix a complete algebraically closed field $C/\Qp$ and an embedding $E \to C$. Fix any open compact subgroup $K^p \sub G(\A_{f}^{p})$. For any open compact subgroup $K_p \sub G(\Qp)$, let $\mathcal{X}_{K^p K_p}^{\ast}$ denote the adic space over $\Spa C$ associated with the base change of the minimal compactification $\Sh_{K^p K_p} (G,X)^{\ast}$ along $E\to C$. Then there is a perfectoid space $\mathcal{X}_{K^p}^{\ast}$ such that 
\[ \mathcal{X}_{K^p}^{\ast} = \varprojlim_{K_p \sub G(\Q_p)}  \mathcal{X}_{K^p K_p}^{\ast, \lozenge}
\] 
as diamonds over $\Spd C$. Moreover, the boundary of $\mathcal{X}_{K^p}^{\ast}$ is Zariski-closed.
\end{theo}

Recall that a Shimura datum $(G,X)$ is of pre-abelian type if there exists a Shimura datum $(G',X')$ of Hodge type admitting an isomorphism of connected Shimura data $(G^{ad}, X^+) \simeq (G'^{ad},X'^+)$. This is slightly more general than the (somewhat more well-known) notion of a Shimura datum of abelian type. While it is probably true that every tower of minimally compactified Shimura varieties with infinite level at $p$ is perfectoid, we expect that Theorem \ref{perfshim} is the most general result which can be proved via current technology.  We also state and prove a similar result for connected Shimura varieties, cf. Theorem \ref{main p-adic}.

We also note that it is not difficult to construct the Hodge-Tate period map $\pi_{\mathrm{HT}}: \mathcal{X}_{K^p}^{\ast} \to \mathscr{F}\! \ell_{G,\mu}$ with all its expected properties, and in fact we did this in the first version of this paper. However, during the revision process we decided to remove this material, since it relied on an unpublished argument of one of us (DH), and also because more general results are now available in recent work of Boxer-Pilloni \cite[Section 4.4]{bp}.
\medskip

While the idea behind the proof of Theorem \ref{perfshim} is clear, the argument is unfortunately somewhat technical.\footnote{A glance at the proof of the key Proposition \ref{going down} should convince the reader of this.} Roughly speaking, there are two key ingredients:
\begin{itemize}
\item ``Perfectoidization results'' \`a la Bhatt--Scholze, building in particular on \cite[Theorem 1.16(1)]{bhatt-scholze}. Roughly speaking, these techniques let us prove that if $(X_i)_{i \in I} \overset{(f_i)_{i \in I}}{\to} (Y_i)_{i \in I}$ is a (pro-)finite morphism between two reasonable inverse systems of rigid analytic spaces, and $\varprojlim_{i \in I} Y_i^{\lozenge}$ is perfectoid, then $\varprojlim_{i \in I} X_i^{\lozenge}$ is also perfectoid. For a precise statement, see Lemma \ref{twotowers}.

\smallskip

\item A new general and user-friendly existence result for quotients of perfectoid spaces by finite groups, cf. Theorem \ref{perfgpquotient}.
\end{itemize}

For open Shimura varieties of abelian type, the problem of proving perfectoidness at infinite level was previously considered by Shen \cite{shen}. We remark that our method is more direct and uses very little from the theory of Shimura varieties and their connected components.

\subsection{Step Two: Topological methods}The second step is totally disjoint from the first, and doesn't use any $p$-adic geometry. We content ourselves with a somewhat impressionistic sketch here. In what follows, assume $G$ is a semisimple group such that $(G,X)$ is a connected Shimura datum of pre-abelian type, and set $d = \dim_{\C} X$ as before.

\medskip

First, we prove an isomorphism of the form $\wt{H}^{i}(K^p) \cong H^i(X_{K^p K_p}, \Map_{cts}(K_p, \Zp))$ for any choice of open compact subgroup $K_p \sub G(\Qp)$.  Here $\Map_{cts}(K_p, \Zp)$ denotes the $K_p$-module of continuous $\Zp$-valued functions on $K_p$. This is essentially a version of Shapiro's lemma, and goes back to a paper of Hill \cite{hi}.  Next, by standard properties of manifolds with boundary, this isomorphism induces an isomorphism $\wt{H}^{i}(K^p) \cong H^i(\ol{X}_{K^p K_p}, \Map_{cts}(K_p, \mathbb{Z}_p))$, where $\ol{X}_{K^p K_p}$ denotes the Borel--Serre compactification of $X_{K^p K_p}$.

\medskip 

By repeated use of excision for compactly supported cohomology, it now suffices to prove that for \emph{some} stratification $\overline{X}_{K^p K_p} = \cup_{Z \in \mathcal{Z}} Z$, we have $H^i_{c}(Z, \Map_{cts}(K_p, \mathbb{Z}_p)|_Z)=0$ for all $i > d$ and all $Z \in \mathcal{Z}$. The key idea can now be phrased as follows:

\medskip

($\dagger$) If we take $\mathcal{Z}$ to be the stratification of $\overline{X}_{K^p K_p}$ obtained by pulling back the usual stratification of $X^{\ast}_{K^p K_p}$ along the map $\pi: \overline{X}_{K^p K_p} \to X^{\ast}_{K^p K_p}$ constructed by Zucker \cite{zu}, then $\mathcal{Z}$ is a stratification with the above property.

\medskip

The idea that ($\dagger$) is both true and provable is perhaps the most novel contribution of this paper; we make some additional remarks on the use of this stratification in Remark \ref{remark on stratifications}. Let us give a sketch of the key ideas. Let $S \sub X^{\ast}_{K^p K_p}$ be a boundary stratum, with preimage $Z= \pi^{-1}(S) \sub \ol{X}_{K^p K_p}$. By the structure theory of the minimal compactification, the strata $S$ are indexed by (equivalence classes of) pairs $(Q, \alpha)$ where $Q \sub G$ is a $\Q$-rational parabolic subgroup whose projection to each simple factor $G_i$ of $G^{ad}$ is either a maximal parabolic or equal to $G_i$, and $\alpha$ is some auxiliary data depending on the level structure. (We will suppress all dependences on level structures in the following discussion.) Moreover, the parabolic $Q$ comes equipped with a canonically defined almost direct product decomposition $Q = U \cdot L \cdot H$. Here $U$ is the unipotent radical of $Q$, $L$ is a reductive group (the \emph{linear} part), $H$ is a semisimple group whose associated symmetric space is Hermitian (the \emph{Hermitian} part); $L\cdot H$ is the full Levi subgroup of $Q$.

\medskip

In parallel with this decomposition of $Q$, the stratum $Z$ almost admits a direct product decomposition $Z \approx Z_U \times Z_L \times Z_H$, where $Z_U$ is a compact nilmanifold, $Z_L$ is the Borel--Serre compactification of a locally symmetric space for the group $L$, and $Z_H \cong S$ is a locally symmetric space for the group $H$. The key idea now is that $H^i_{c}(Z, \Map_{cts}(K_p, \mathbb{Z}_p)|_Z)$ can also be decomposed accordingly, by a K\"unneth-like formula, into contributions coming from each of these three factors, which can each be controlled:
 
\begin{itemize}
 \item The contribution of $Z_U$ is trivial, which follows from a well-known vanishing principle for completed cohomology of unipotent groups.
 
 \item The contribution of $Z_H$ can be expressed in terms of compactly supported completed cohomology for $H$, which can be controlled by Step One.
 
 \item The contribution of $Z_L$ can be expressed in terms of completed cohomology for $L$, which can be controlled using the bounds in \cite{bs}, or even using the trivial bound.
 \end{itemize}
 
 The critical observation here is that Step One gives such good control over the contribution of $Z_H$ that we need very little control over the contribution of $Z_L$.
 
\medskip
 
In reality, the above sketch is somewhat oversimplified, because $Z$ does not really admit a direct product decomposition; rather, it has the structure of an iterated fibration whose fibers are as described above. This makes the proof somewhat more complicated. Nevertheless, the essential idea follows the outline given above.

\medskip

Let us briefly outline the contents of this paper. Section \ref{top preliminaries} collects some preliminaries on topology and arithmetic groups that are needed for the computations in later section. Section \ref{sec: cc} discusses completed (co)homology and the Calegari--Emerton conjectures, carrying out the core of ``Step Two'' above. Section \ref{sec: shimura} introduces Shimura varieties and proves our main results on the Calegari--Emerton conjectures, including Theorem \ref{maintheorem}. Section \ref{sec: perfectoid} carries out ``Step One'', proving Theorem \ref{perfshim} and deducing the vanishing theorem for compactly supported completed cohomology. We note that section \ref{sec: perfectoid} is completely independent of the previous sections. Conversely, the vanishing result Corollary \ref{compactly supported vanishing} is the only part of section \ref{sec: perfectoid} that gets used in previous sections.

\subsection*{Acknowledgments} We would like to thank Bhargav Bhatt, Frank Calegari, Ana Caraiani, Matt Emerton, Michael Harris, Ben Heuer, Kai-Wen Lan, Vincent Pilloni, Peter Scholze, and Jack Thorne for helpful conversations related to the material in this paper. We also extend a special thanks to Mark Goresky, whose survey \cite{goresky} was very helpful in the initial stages of the project. Moreover, we gratefully acknowledge the Herchel Smith Foundation and the Max Planck Institute for Mathematics in Bonn for supporting visits to Cambridge and Bonn, respectively, during which work on this project was carried out. Finally, we are grateful to the referee for their careful reading and very helpful comments and corrections.  C.J. was supported by the Herchel Smith Foundation during part of this project.

\section{Preliminaries}\label{top preliminaries}

In this section we collect some facts and definitions from topology and algebraic groups that we will need. We make no attempt to state results in maximal generality and none of them are original, but we have often had difficulties locating the precise statements that we need in the literature. We hope that collecting this material here is of sufficient aid to the reader to justify its inclusion.

\medskip

The topological spaces that we will work with will mostly be smooth manifolds with boundary; we will simply write ``manifold with boundary'' to mean a smooth manifold with boundary. Any smooth manifold with boundary admits a combinatorial triangulation for which the boundary is a subsimplicial complex (see e.g. \cite[Theorem 10.6]{mun}). We recall that if $\ol{X}$ is  a manifold with boundary with interior $X$ and $U\sub \ol{X}$ is an open subset containing $X$, then the inclusion $j :U \to \ol{X}$ is homotopy equivalence by the global collar neighbourhood theorem. In a very similar vein, if $\cF$ is a local system on $\ol{X}$, a simple local calculation shows that $Rj_\ast j^{-1}\cF = \cF$. In particular, we obtain canonical isomorphisms $H^i(U,\cF) \cong H^i(\ol{X},\cF)$ which we will often treat as equalities.

\medskip

All actions of groups on topological spaces will be left actions in this section. Of course, all results have natural analogues for right actions (and we will use them).

\subsection{Local systems}

Let $X$ be a topological space and let $\G$ be a group acting from the left on $X$. In this paper most of our actions will be \emph{free}\footnote{The most common terminology for this notion seems to be a free and properly discontinuous action, but we find this terminology rather cumbersome.}, by which we mean that every point $x\in X$ has an open neighborhood $U$ such that $U \cap \gamma U \neq \emptyset$ only if $\gamma =1$. The quotient map $\pi : X \to X_\G := \G \backslash X$ is then a covering map, and we recall that any left $\G$-module\footnote{By which we always mean a (left) $\Z[\G]$-module, unless otherwise stated.} $M$ defines a local system $\wt{M}$ on $X_\G$ given by 
\[
\wt{M}(U) = \Map_{lc,\G}(\pi^{-1}(U),M)
\]
where the right hand side denotes the locally constant functions $f : \pi^{-1}(U) \to M$ satisfying $f(\gamma x) = \gamma . f(x)$ for all $\gamma \in \G$ and all $x \in \pi^{-1}(U)$. When $X$ is a manifold with boundary, this may be written as
\[
\wt{M}(U) = \Map_{\G}(\pi_0(\pi^{-1}(U)),M),
\]
where $\Map$ simply denotes set-theoretic functions (as $\pi_0(\pi^{-1}(U))$ is discrete). The following theorem is well known, and follows directly from the fact that the singular chain complex $C_\bu(X)$ is a resolution of $\Z$ by free $\G$-modules.

\begin{theo}\label{singular and group coh}
Let $X$ is a contractible manifold with boundary with a free action of $\G$. Then
\[
H^\ast(X_\G,\wt{M}) \cong \Ext^\ast_{\Z[\G]}(\Z,M) \cong H^\ast(\G,M)
\]
canonically for every $\G$-module $M$.
\end{theo}

We now consider a relative version of Theorem \ref{singular and group coh}. Let $p : E \to B$ be a fibre bundle with contractible fibre $F$ (all spaces are manifolds with boundary). Assume that we have a group $\G$ acting (from the left) on both $E$ and $B$, making $p$ $\G$-equivariant. We assume further that the action of $\G$ is free on $E$, and that the action of $\G$ on $B$ factors through a quotient $\Delta$ which acts freely on $B$. Set $N = \Ker(\Gamma \to \Delta)$; $N$ then acts freely on the fibres of $p$. Consider the induced map
\[
q : E_\G \to B_\Delta
\]
on quotients. 

\begin{coro}\label{formula for pushforward}
Let $M$ be a $\Gamma$-module and let $i\geq 0$. Then $R^i q_\ast \wt{M}$ is the local system on $B_\Delta$ given by the $\Delta$-module $H^i(N,M)$.
\end{coro}

\begin{proof}
We begin by proving the case $i=0$. Write $\pi_E : E \to E_\G$ and $\pi_B : B \to B_\Delta$ for the quotient maps and let $U \sub B$ be open. From the definitions, one sees that 
\[
q_\ast \wt{M} (U) = \Map_{lc,\G}(p^{-1}\pi_B^{-1}(U),M).
\]
Since the fibres of $p$ are connected and the action of $N$ preserves the fibres, we have $\Map_{lc,\G}(p^{-1}\pi_B^{-1}(U),M) = \Map_{lc,\Delta}(\pi_B^{-1}(U),M^N)$, which is the desired statement.

\medskip

This proves that the diagram of functors
\[
    \xymatrix{ \mathrm{Mod}_{\G} \ar[r] \ar[d]^{M \mapsto M^N} & \mathrm{Sh}(E_\Gamma) \ar[d]^{q_\ast} \\ \mathrm{Mod}_{\Delta} \ar[r] & \mathrm{Sh}(B_\Delta) }
\]
commutes up to natural isomorphism, where the horizontal functors are the local systems functors $M \mapsto \wt{M}$. The horizontal functors are exact (by looking at stalks), so it suffices to show that $M \mapsto \wt{M}$ sends injective $\G$-modules to $q_\ast$-acyclic sheaves on $E_\G$ (then the diagram above commutes also after passing to derived categories and derived functors, which is what we want).

\medskip

So, let $M$ be an injective $\G$-module, and let $i\geq 1$. $R^i q_\ast \wt{M}$ is the sheafification of the presheaf $U \mapsto H^i(q^{-1}(U),\wt{M}) $ on $B_\Delta$. There is a basis of open subsets $U$ of $B_\Delta$ which are contractible and for which the fibre bundle $q : q^{-1}(U) \to U$ is trivial. In this case $q^{-1}(U) \cong U \times N \backslash F$ and hence 
\[
H^i(q^{-1}(U),\wt{M}) \cong H^i(N,M)
\]
by Theorem \ref{singular and group coh}. But $M$ is an injective $N$-module since the restriction functor from $\G$-modules to $N$-modules has an exact left adjoint $V \mapsto \Z[\G] \otimes_{\Z[N]} V$. Thus $H^i(q^{-1}(U),\wt{M})=0$ for all such $U$, and hence $R^i q_\ast \wt{M}=0$ as desired.
\end{proof}

We will also use a (less precise but more general) version for pushforwards with proper support.

\begin{prop}\label{formula for proper pushforward}
Let $f : X \to Y$ be a fibre bundle of manifolds with boundary, with fibre $Z$ (also a manifold with boundary). Let $\cF$ be a local system on $X$. Then, for any $i\geq 0$, $R^i f_! \cF$ is a local system on $Y$ with fibre $H_c^i(Z,\cF)$. 
\end{prop}

\begin{proof}
We will use the commutation of derived pushforward with proper support with (arbitrary) pullbacks; see \cite[Proposition 2.6.7]{ks}. Let $U \sub Y$ be a contractible open subset such that $f$ is trivial over $U$, i.e. isomorphic to the canonical projection $p_U : U \times Z \to U$. These form an open cover of $Y$, so since $R^i f_!$ commutes with pullback it suffices to show that $R^i p_{U,!}\cF$ is a constant sheaf. Since $U$ is contractible, the restriction of $\cF$ to $U\times Y$ $\cF$ comes by pullback from a local system on $Y$, which we will call $\cF_Z$. Consider the cartesian diagram
\[
    \xymatrix{ U \times Z \ar[r]^{p_Z} \ar[d]^{p_U} & Z \ar[d]^{g} \\ U \ar[r]^{f} & pt, }
\]
where $pt$ denotes the point and $f$ and $g$ are the canonical maps. Then we have
\[
R^i p_{U,!}\cF = R^i p_{U,!}p_Z^{-1} \cF_Z \cong f^{-1}R^i g_! \cF_Z. 
\]
In other words, $R^i p_{U,!}\cF$ is the pullback of $H^i_c(Z, \cF_Z)$ via the canonical map $U \to pt$. This proves the proposition.
\end{proof}

Next, let $X$ be a manifold with boundary with a free left action of a group $\G$, and assume that $\G^\prime \sub \G$ is a finite index subgroup. Consider the natural map $q : X_{\G^\prime} \to X_\G$. If $M$ is a $\G^\prime$-module, we put 
\[
\Ind_{\G^\prime}^\G M = \{ f : \G \to M \mid f( \gamma^\prime \gamma) = \gamma^\prime . f(\gamma) \,\,\, \forall \gamma^\prime \in \G^\prime, \, \gamma \in \G \},
\]
which is a left $\G$-module under right translation $(\gamma.f)(x) = f(x\gamma)$. We then have the following.

\begin{prop}\label{pushforward at finite level}
With notation and assumptions as above, $R^i q_\ast \wt{M} = 0$ for $i \geq 1$, and $q_\ast \wt{M}$ is the local system attached to $\Ind_{\G^\prime}^\G M$.
\end{prop}

\begin{proof}
The map $q$ is proper, so if $x\in X_\G$, then $(R^i q_\ast \wt{M})_x = H^i(q^{-1}(x),\wt{M})$ (by \cite[Proposition 2.6.7]{ks}), and $q^{-1}(x)$ has no higher cohomology since it is a finite set. This proves the first part. To compute $q_\ast \wt{M}$, let $U \sub X_\G$ be open and write $\pi : X \to X_\G$ for the quotient map. Unwinding the definitions, we see that 
\[
q_\ast \wt{M}(U) = \Map_{\G^\prime}(\pi_0(\pi^{-1}(U)),M),
\]
and the right hand side is easily seen to be equal to $\Map_{\G}(\pi_0(\pi^{-1}(U)),\Ind_{\G^\prime}^\G M)$ functorially in $U$, as desired.
\end{proof}

We move on to results on the commutation of $M \mapsto \wt{M}$ with direct limits. First, let $\ol{X}$ be a manifold with boundary, with a free left action of a group $\G$. Write $\ol{X}_\G := \G \backslash \ol{X}$; we assume that $\ol{X}_\G$ is compact, so it has a finite triangulation. Fix such a triangulation and pull it back to $\ol{X}$; this gives a triangulation whose corresponding complex of simplicial chains $C_\bu^\Delta(\ol{X})$ is a bounded complex of finite free $\Z[\G]$-modules. Let $(M_i)_{i\in I}$ be a directed system of $\G$-modules with direct limit $M = \varinjlim_i M_i$.

\begin{lemm}\label{direct limits}
The natural map
\[
\varinjlim_i H^\ast(\ol{X}_\G,\wt{M}_i) \to H^\ast(\ol{X}_\G,\wt{M})
\]
is an isomorphism.
\end{lemm} 

\begin{proof}
The canonical map
\[
i : C_\bu^\Delta(\ol{X}) \to C_\bu(\ol{X})
\]
is $\G$-equivariant and a quasi-isomorphism; since the terms of both complexes are projective $\Z[\G]$-modules the map is therefore a chain homotopy equivalence. This then gives us a commutative diagram of complexes
\[
  \xymatrix{ \varinjlim_i \Hom_\G(C_\bu(\ol{X}),M_i) \ar[r] \ar[d] & \Hom_\G(C_\bu(\ol{X}),M) \ar[d] \\ \varinjlim_i \Hom_\G(C^\Delta_\bu(\ol{X}),M_i) \ar[r] & \Hom_\G(C^\Delta_\bu(\ol{X}),M) }
\]
where the vertical maps are induced by $i$ and the horizontal maps are the natural maps. The vertical maps are then quasi-isomorphisms since they are induced from $i$, and the lower horizontal map is an isomorphism since $C^\Delta_\bu(\ol{X})$ is bounded complex of finite free $\Z[\G]$-modules. The top horizontal map is therefore a quasi-isomorphism as well, and taking cohomology gives the desired result.
\end{proof}

We can then prove the result in greater generality. With $\ol{X}$ and $\G$ as above, let $U\sub \ol{X}$ be a $\G$-invariant open subset containing the interior of $\ol{X}$. Set $U_\G := \G \backslash U$, $Z := \ol{X}\setminus U$ and $Z_\G := \G \backslash Z$. 

\begin{prop}\label{direct limits general}
The natural map
\[
\varinjlim_i H_?^\ast(U_\G,\wt{M}_i) \to H_?^\ast(X_\G,\wt{M})
\]
is an isomorphism, for $?\in \{\emptyset,c\}$.
\end{prop} 

\begin{proof}
For $?=\emptyset$ this reduces directly to Lemma \ref{direct limits} by our setup, so assume that $?=c$. By naturality of the excision sequence and exactness of direct limits we have a commutative diagram
\[
  \xymatrix{ \dots \ar[r] & \varinjlim_i H_c^j(U_\G,\wt{M}_i) \ar[r] \ar[d] & \varinjlim_i H^j(\ol{X}_\G,\wt{M}_i) \ar[r] \ar[d] & \varinjlim_i H^j(Z_\G,\wt{M}_i) \ar[r] \ar[d] & \dots \\ \dots \ar[r] & H_c^j(U_\G,\wt{M}) \ar[r] & H^j(\ol{X}_\G,\wt{M}) \ar[r] & H^j(Z_\G,\wt{M}) \ar[r] & \dots }
\]
with exact rows. The result then follows from Lemma \ref{direct limits} (since it is applicable to both $\ol{X}_\G$ and $Z_\G$) and the five lemma.
\end{proof}

\subsection{``Completed cohomology''} In this subsection we make some definitions and recall a theorem of Hill which we will use to handle completed cohomology later. To begin with, we make the following general definition. Let $R = \varprojlim_i R/\mathfrak{a}^n $ be an adic ring, with $\mathfrak{a}$ a finitely generated ideal of definition. 

\begin{defi}
Let $(X_i)_{i\in I}$ be an inverse system of topological spaces, with inverse limit $X$. We define the completed cohomology groups $\wt{H}_?^\ast(X,R)$ of $(X_i)_{i\in I}$ with coefficients in $R$, to be
\[
\wt{H}_?^\ast(X,R) = \varprojlim_n \varinjlim_i H_?^\ast(X_i,R/\mathfrak{a}^n).
\]
Here $?\in \{\emptyset,c\}$, i.e. we consider either usual or compactly supported cohomology, when the latter makes sense.
\end{defi}

\begin{rema}
A few remarks on this definition:
\begin{enumerate}

\item The notation is chosen for simplicity; we make no assertion that $\wt{H}_?^\ast(X,R)$ only depends on $X$. One weak form of independence is clear though: We may replace $I$ with a cofinal subsystem $J$. In particular, we may always assume that $I$ contains an initial element $0\in I$.

\medskip

\item We will almost exclusively work with discrete $R$, where the inverse limit in the definition of $\wt{H}_?^\ast(X,R)$ disappears.
\end{enumerate}
\end{rema}

We now recall the computation of completed cohomology as the cohomology of a ``big'' local system at finite level in some circumstances, which first appeared in \cite{hi}. Let $\ol{X}$ be a manifold with boundary, equipped with a left action of a group $G$. We assume that there is a subgroup $\G \sub G$ which acts freely on $\ol{X}$, and suppose that $\G = \G_0 \supseteq \G_1 \supseteq \G_2 \supseteq \dots $ is a sequence of finite index subgroups of $\G$. Let $X \sub \ol{X}$ be a $\G$-stable open subset containing the interior of $\ol{X}$. Set
\[
\wh{X}:= \varprojlim \left( \dots \to \G_2 \backslash X  \to \G_1 \backslash X \to \G_0 \backslash X \right).
\]
and define
\[
K = \varprojlim_i  \G_i \backslash \G;
\]
this is a profinite set with a right action of $\G$. Assume that $\ol{X}_\Gamma$ is compact. Then we get the following formula for completed cohomology of $\wh{X}$ (cf. \cite[Corollary 1]{hi}):

\begin{prop}\label{cc at finite level}
With assumptions as above, let $R$ be a discrete ring and let $?\in\{\emptyset,c\}$. Then there is a canonical isomorphism
\[
\tH_?^\ast(\wh{X},R) \cong H^\ast_?\left( X_\Gamma, \wt{\Map_{cts}(K,R)} \right), 
\] 
where $\G$ acts on $\Map_{cts}(K,R)$ via right translation.
\end{prop}

\begin{proof}
By Lemma \ref{pushforward at finite level} and the definition, we have
\[
\tH_?^\ast(\wh{X},R) \cong \varinjlim_i H_?^\ast\left( X_\Gamma, \wt{\Map(\G_i \backslash \G, R)} \right). 
\]
Our setup implies that we may apply Proposition \ref{direct limits general} to the right hand side, so it remains to show that
\[
\varinjlim_i \Map(\G_i \backslash \G, R) = \Map_{cts}(K,R)
\]
as $\G$-modules. But this is immediate from the definition of $K$.
\end{proof}

We will also encounter local systems slightly bigger than the one appearing in Proposition \ref{cc at finite level}. We keep the notation and assumptions of Proposition \ref{cc at finite level}, except that we forget the groups denoted by $K$ and $\G_i$, $i\geq 1$. Let $G$ be a profinite group with closed subgroups $K \sub H \sub G$, and assume that $K$ is normal in $H$. For simplicity, we assume that there is a countable basis of neighborhoods of $1\in G$. Suppose that we have a group homomorphism $\G \to H/K$; then $\Map_{cts}(H/K,R)$ and $\Map_{cts}(G/K,R)$ become left $\G$-modules via right translation, and hence induce local systems on the space $X_\G$. Then we have the following simple but useful lemma.

\begin{lemm}\label{big and small local systems}
Fix an integer $q\geq 0$ and let $? \in \{\emptyset,c\}$. 
\begin{enumerate}

\item $H_?^q(X_\G, \wt{\Map_{cts}(H/K,R)}) = 0$ if and only if $H_?^q(X_\G, \wt{\Map_{cts}(G/K,R)}) = 0$;

\item $H_c^q(X_\G, \wt{\Map_{cts}(H/K,R)}) \to H^q(X_\G, \wt{\Map_{cts}(H/K,R)})$ is injective (or surjective, or bijective) if and only if $H_c^q(X_\G, \wt{\Map_{cts}(G/K,R)}) \to H^q(X_\G, \wt{\Map_{cts}(G/K,R)})$ is injective (or surjective, or bijective).

\end{enumerate}
\end{lemm}

\begin{proof}
Choose a continuous splitting of the natural map $G/K \to G/H$ (the existence of which is easy to prove using the assumption that $1\in G$ has a countable basis of neighborhoods); this gives a homeomorphism
\[
G/K \cong G/H \times H/K 
\]
of right $H/K$-spaces (where $H/K$ acts on the right hand side through the second factor). Then
\[
\Map_{cts}(G/K,R) \cong \Map_{cts}(G/H \times H/K ,R)\cong \Map_{cts}(G/H,R) \otimes_R \Map_{cts}(H/K,R)
\]
as $H/K$-modules (and hence as $\G$-modules), where the action is trivial on $\Map_{cts}(G/H,R)$. Now $\Map_{cts}(G/H,R)$ is a direct limit of finite free $R$-modules, so using Proposition \ref{direct limits general} we have an isomorphism 
\[
H_?^q(X_\G, \wt{\Map_{cts}(G/K,R)}) \cong \Map_{cts}(G/H,R) \otimes_R H_?^q(X_\G, \wt{\Map_{cts}(H/K,R)})
\]
which respects the maps in part (2). The lemma follows from this (since $\Map_{cts}(G/H,R)$ is a free $R$-module).
\end{proof}

\subsection{Arithmetic and congruence subgroups}

Here we quickly recall some material on arithmetic and congruence subgroups. Let $G$ be a connected linear algebraic group over $\Q$. Congruence subgroups of $G(\Q)$ are subgroups of the form $G(\Q) \cap K$, where $K \sub G(\A_f)$ is a compact open subgroup and the intersection is taken inside $G(\A_f)$. A subgroup in $G(\Q)$ is usually said to be arithmetic if it is commensurable with one (equivalently any) congruence subgroup. In fact, one can require a slightly stronger condition.

\begin{prop}\label{arithmetic contained in congruence}
Any arithmetic subgroup in $G(\Q)$ is contained in a congruence subgroup.
\end{prop}

\begin{proof}
Let $\G$ be an arithmetic subgroup, let $K_1 \sub G(\A_f)$ be a compact open subgroup and set $\G_1 = G(\Q) \cap K_1$. The closure of $\G_1$ in $G(\A_f)$ is compact since $\G_1 \sub K_1$, and since $\G$ and $\G_1$ are commensurable this easily implies that the closure of $\G$ in $G(\A_f)$ is a compact subgroup. Since any compact subgroup of locally profinite group is contained in a compact open subgroup, we deduce the existence of a compact open subgroup $K_2 \sub G(\A_f)$ with $\G \sub K_2$. It follows that $\G$ is contained in the congruence subgroup $G(\Q) \cap K_2$, as desired.
\end{proof}

Moving on, let $H$ be another connected linear algebraic group, and let $\G \sub G(\Q)$ be an arithmetic subgroup. If $H \sub G$ is a subgroup, then directly from the definitions we see that $\G \cap H(\Q)$ is an arithmetic subgroup in $H(\Q)$, which is congruence if $\G$ is. If we instead have a surjection $f : G \to H$, then $f(\G)$ is an arithmetic subgroup (see \cite[Theorem 4.1]{pr}); this will be important in this paper and we will use it freely. Before moving on, we recall that group cohomology for any torsion-free arithmetic subgroup $\G$ commutes with direct limits.

\medskip

We recall the notion of neatness from \cite[\S 17.1]{bo2}. An element $\gamma\in G(\Q)$ is called neat if there is a faithful representation $r : G \to \GL(V)$ such that the multiplicative group generated by the eigenvalues of $r(\gamma)$ (in one, or equivalently any, algebraically closed field containing $\Q$) is torsion-free. A neat element cannot have finite order. An arithmetic subgroup $\G \sub G(\Q)$ is called neat if all its elements are neat; such subgroups are in particular torsion-free. From the definitions, we see that if $H \sub G$ is a connected linear algebraic subgroup and $\G \sub G(\Q)$ is neat, then $\G \cap H(\Q)$ is neat. If an element $\gamma$ is neat, then for any representation $\rho : G \to GL(W)$, the subgroup generated by the eigenvalues of $\rho(\gamma)$ is torsion-free \cite[Corollaire 17.3]{bo2}. An easy consequence of this is that if $f : G \to H$ is a surjection of linear algebraic groups and $\G$ is neat, then $f(\G)$ is neat.

\medskip

For language reasons, let us also introduce notions of neatness for adelic and $p$-adic groups. The notion of neatness for an element $g=(g_p)_p \in G(\A_f)$ and a subgroup $K\sub G(\A_f)$ is defined in \cite[\S 0.6]{pink-thesis}. For $p$-adic groups, we make the definition analogous to the case of arithmetic groups: An element $g\in G(\Qp)$ is called neat if there is a faithful representation $\rho : G_{\Qp} \to \GL(W)$ over $\Qp$ such that the multiplicative group generated by the eigenvalues of $\rho(g)$ (in one, or equivalently any, algebraically closed field containing $\Qp$) is torsion-free. Again, this is independent of the choice of $\rho$. A subgroup $K_p \sub G(\Qp)$ is called neat if all of its elements are neat. We note the following implications among these concepts: If $K_p \sub G(\Qp)$ is a neat compact open subgroup, then $K^p K_p \sub G(\A_f)$ is neat for any compact open $K^p \sub G(\A_f^p)$. If a compact open $K\sub G(\A_f)$ is neat, then $\G = \G(\Q) \cap K$ is a neat congruence subgroup of $G$.

\medskip

We record the following version of the standard result that ``sufficiently small'' congruence subgroups are neat; it will be important for us to be able to only impose congruence conditions at a fixed prime $p$.

\begin{prop}\label{existence of neat subgroups}
Let $p$ be a prime. Then sufficiently small compact open subgroups of $G(\Qp)$ are neat. In particular, if $K^p \sub G(\A^p_f)$ is compact open, then $K=K^pK_p$ and $\G = G(\Q) \cap K$ are neat for sufficiently small $K_p\sub G(\Qp)$.
\end{prop}

\begin{proof}
By choosing a faithful representation $\rho : G \to \GL_n$ (and remembering that any compact subgroup of a locally profinite group is contained in a compact open subgroup), we may reduce to $G=\GL_n$. In this case, set $K_{r,p} = \Ker( \GL_n(\Zp) \to GL_n(\Z/p^r))$; we will prove that if $r>n/(p-1)$, then $K_{r,p}$ is neat, so we assume this condition on $r$ from now on. To show neatness, it suffices to show that if $\gamma \in K_{r,p}$, then the group generated by the eigenvalues of $\gamma$ is torsion-free. Let $\alpha_1,\dots,\alpha_n$ be the eigenvalues of $\gamma$ (in some choice of $\ol{\Q}_p$, with valuation $v_p$ normalized so that $v_p(p)=1$). The characteristic polynomial of $\gamma$ reduces to $(X-1)^n$ modulo $p^r$, so by looking at Newton polygons $v_p(\alpha_i -1) \geq r/n$ for all $i$. Thus, if $\alpha$ is any element in the multiplicative group generated by the $\alpha_i$, $v_p(\alpha-1) \geq r/n$. In particular, since $r > n/(p-1)$, $\alpha$ cannot be a nontrivial root of unity.  This finishes the proof of the proposition.
\end{proof}

We also recall another fact about ``sufficiently small'' congruence subgroups, and set up some notation. For any real Lie group $J$, we write $J^+$ for the identity component of $J$. The following is \cite[Corollaire 2.0.14]{de}.

\begin{prop}\label{deligne}
Let $G$ be a connected reductive group over $\Q$. Then there exists a congruence subgroup $\G \sub G(\Q)$ which is contained in $G(\R)^+$. In particular, if $\Delta \sub G(\Q)$ is any congruence subgroup, then $\Delta \cap G(\R)^+$ is also a congruence subgroup.
\end{prop}

We remark that, unlike neatness, the condition $\G \sub G(\R)^+$ cannot be enforced only by congruence conditions at a single prime (chosen independently of $G$). For a simple example, consider $G= \mathrm{Res}_{\Q}^F \mathbb{G}_m$ with $F:= \Q(\sqrt{3})$, and consider the totally negative unit $\alpha = -2+ \sqrt{3} \in F$. One checks easily that $\alpha^{3^n} \equiv 1$ modulo $3^n$ for all $n$ but all the $\alpha^{3^n}$ are totally negative. For an example with a semisimple $G$, consider $G = \mathrm{Res}_{\Q}^F \mathrm{PGL}_2$ and the matrices
\[
\begin{pmatrix}
\alpha^{3^n} & 0 \\ 0 & 1
\end{pmatrix}, \,\,\,\, n \geq 1;
\]
again these tend to the identity $3$-adically but they all lie in a non-identity component since they have totally negative determinant.

\subsection{Cohomology of unipotent groups}\label{coh of unipotent groups}

From now on we fix a prime number $p$. Let $N$ be a unipotent algebraic group over $\Q$. The goal in this subsection is to prove the following theorem (we remark that $N$ satisfies strong approximation and that all arithmetic subgroups of $N(\Q)$ are congruence subgroups):

\begin{theo}\label{comparison of cts and disc gp coh} If $\G \sub N(\Q)$ is a congruence subgroup with closure $K_p \sub N(\Qp)$ and $V$ is a smooth $K_p$-representation over $\Fp$, then the natural map
\[
H^i_{cts}(K_p,V) \to H^i(\G,V)
\]
is an isomorphism for all $i$. 
\end{theo}

We start with some recollections. First, in the situation above, $\G = N(\Q) \cap K$ for some open compact subgroup $K \sub N(\A_f)$, and $\G$ is dense in $K$ by strong approximation for $N$. In particular, $K_p$ is the image of $K$ under the projection map $N(\A_f) \to N(\Qp)$, and hence open. We have a natural forgetful functor
\[
{\rm Mod}_{sm}(K_p,\Fp) \to {\rm Mod}(\G)
\]
and if $V \in {\rm Mod}_{sm}(K_p,\Fp)$, then $V^\G = V^{K_p}$ by smoothness of $V$ and density of $\G$ in $K_p$. In light of this, Theorem \ref{comparison of cts and disc gp coh} follows directly from the following special case, which is in fact all we will need.

\begin{prop}\label{injectives are acyclic}
Let $V$ be an injective smooth $K_p$-representation over $\Fp$. Then $H^i(\G,V)=0$ for all $i\geq 1$.
\end{prop}

We will prove this by induction on $\dim N$. Before the main argument, we will discuss the structure of injective $K_p$-representations. Let $W$ be any $\Fp$-vector space, which we give the discrete topology. We can form $\Map_{cts}(K_p, W)$, where $K_p$ acts by right translation. This is the smooth induction of $W$, viewed as a representation of the trivial group, to $K_p$. Since smooth induction has an exact left adjoint (restriction), $\Map_{cts}(K_p, W)$ is injective for any $W$. We will refer to these representations as ``standard injectives''. Now if $V \in {\rm Mod}_{sm}(K_p)$ is arbitrary, there is a $K_p$-equivariant injection
\[
V \to \Map_{cts}(K_p,V)
\]
given by $v \mapsto (k \mapsto kv)$, where $K_p$ acts on $\Map_{cts}(K_p,V)$ by right translation. Thus there are enough standard injectives, and any injective is a direct summand of a standard injective. In particular, it suffices to prove Proposition \ref{injectives are acyclic} for standard injectives. Moreover, since group cohomology of $\G$ commutes with direct limits, it suffices to prove Proposition \ref{injectives are acyclic} for $\Map_{cts}(K_p,\Fp)$.

\medskip

We now begin the induction. First assume that $\dim N =1$, i.e. that $N=\mb{G}_a$. Then (up to isomorphism) $\G=\Z$ and $K_p =\Zp$. There are a number of ways of proving that $H^i(\Z, \Map_{cts}(\Zp,\Fp))=0$ for $i\geq 1$. For example, by Proposition \ref{cc at finite level},
\[
H^i(\Z, \Map_{cts}(\Zp,\Fp)) = \varinjlim_n H^i(\R/p^n\Z, \Fp) = \varinjlim_n H^i(S^1, \Fp)
\]
where on the right the transition maps come from pullback along the maps $S^1 \to S^1$, $z \mapsto z^p$. All groups are $0$ for $i\geq 2$, and for $i=1$ one easily sees that the transition maps are all $0$, so this proves Proposition \ref{injectives are acyclic} for $N = \mb{G}_a$. 

\medskip

We move on to the induction step. By the structure of unipotent groups, we can choose a proper non-trivial normal subgroup $U \sub N$. Set $H=N/U$ and let $f : N \to H$ denote the natural map. Put $\G_U = \G \cap U(\Q)$, $\G_H = f(\G)$, $K_{U,p} = K_p \cap U(\Qp)$ and $K_{H,p} = f(K_p) \sub H(\Qp)$. Then $K_{U,p}$ is the closure of $\G_U$ in $U(\Qp)$ and $K_{H,p}$ is the closure of $\G_H$ in $H(\Qp)$. Let $V$ be an injective smooth $K_p$-representation over $\Fp$. We have the Hochschild--Serre spectral sequence
\[
H^i(\G_H, H^j(\G_U, V)) \implies H^{i+j}(\G,V).
\]
The restriction of $V$ to $K_{U,p}$ is still injective by \cite[Proposition 2.1.11]{em-ord2}. Thus, by the induction hypothesis, $H^j(\G_U,V)=0$ for $j\geq 1$, and hence the spectral sequence degenerates to $H^i(\G, V) = H^i(\G_H, V^{\G_U})$. By above, $V^{\G_U}=V^{K_{U,p}}$, which is an injective\footnote{$M \mapsto M^{K_{U,p}}$ preserves injectives, since inflation from $K_{H,p}$ to $K_p$ provides an exact left adjoint.} $K_{H,p}$-module. By the induction hypothesis again we get
\[
H^i(\G, V) = H^i(\G_H, V^{K_{U,p}}) = 0
\]
for $i\geq 1$, as desired. This finishes the proof of Proposition \ref{injectives are acyclic}, and hence the proof of Theorem \ref{comparison of cts and disc gp coh}.  

\section{Completed cohomology of locally symmetric spaces}\label{sec: cc}

We continue to fix a prime number $p$.

\subsection{Locally symmetric spaces}\label{bs section}

In this section we recall some material on locally symmetric spaces and their Borel--Serre compactifications. Let $G$ be a connected linear algebraic group over $\Q$, let $A=A_G\sub G$ be a maximal torus in the $\Q$-split part of the radical of $G$ and let $K_\infty = K_{G,\infty}\sub G(\R)$ be a maximal compact subgroup. We will work with the (connected) symmetric space
\[
X = X^G := G(\R)^+/A(\R)^+K_\infty^+ = G(\R)/A(\R)K_\infty,
\]
which is the symmetric space part of any space of type $S-\Q$ for $G$, in the terminology of \cite{bs}. If $\G \sub G(\Q)$ is a torsion-free arithmetic subgroup, then $\G$ acts freely on $X$ and the quotient $\G\backslash X$ is a locally symmetric space. If $K \sub G(\A_f)$ is a compact (not necessarily open) subgroup, we will set
\[
X^G_K := G(\Q)^+ \backslash X \times G(\A_f) / K,
\]
where $G(\Q)^+ := G(\Q) \cap G(\R)^+$ and $K$ and $G(\A_f)$ carry their usual adelic topologies. When $K$ is additionally open and $g\in G(\A_f)$, set $\G_g = \G_{g,K} := G(\Q)^+ \cap gKg^{-1} $; these are congruence subgroups by Proposition \ref{deligne}. We have the following decomposition as topological spaces
\[
X^G_K \,\,\cong \bigsqcup_{g \in G(\Q)^+ \backslash G(\A_f) /K } \G_g \backslash X, 
\] 
where the set $\Sigma_K := G(\Q)^+\backslash G(\A_f) / K$ is finite by \cite[Theorem 5.1]{bo}. If $K$ is neat, then all the $\G_g$ are neat and in particular torsion-free, so $X^G_K$ is a (possibly disconnected) manifold of dimension $\dim_\R X$.

\medskip

Recall the Borel--Serre bordification $\ol{X}=\ol{X}^G$ of $X=X^G$ from \cite{bs}. $\ol{X}$ has a natural structure of a manifold with corners, with interior $X$. We write $\partial X = \ol{X}\setminus X$. The action of $G(\Q)$ on $X$ extends to an action of $\ol{X}$, and again any torsion-free arithmetic subgroup $\G \sub G(\Q)$ acts freely on $\ol{X}$. As a set,
\[
\ol{X} = \bigsqcup_Q X^Q
\] 
where $Q$ runs through the (rational) parabolic subgroups of $G$. The closure of $X^Q$ inside $\ol{X}$ is $\ol{X}^Q = \bigsqcup_{P^\pr \sub Q} X^{P^\pr}$. Write $C_Q$ for the set of parabolics $Q^\pr$ of $G$ which are conjugate to $Q$ (over $\Q$); $C_Q$ carries a (tautological) left $G(\Q)$-action by conjugation. Fix a minimal parabolic $P$ of $G$ over $\Q$ for simplicity. We can then write
\[
\ol{X} = \bigsqcup_Q X^Q = \bigsqcup_{Q \supseteq P} \bigsqcup_{Q^\pr \in C_Q} X^{Q^\pr},
\]
and the subsets $X^{G,Q} = \bigsqcup_{Q^\pr \in C_Q} X^{Q^\pr}$ are stable under $G(\Q)$. If $g\in G(\Q)$, then $gX^{Q^\pr} = X^{gQ^\pr g^{-1}}$ and hence the stabilizer of $X^{Q^\pr}$ is $Q^\pr(\Q)$. In particular, if $\G \sub G(\Q)$ is an arithmetic subgroup, we see that
\[
\G \backslash \ol{X} = \bigsqcup_{Q \supseteq P} \bigsqcup_{Q^\pr \in C_{Q,\Gamma}} \Gamma_{Q^\pr} \backslash X^{Q^\pr}, 
\]
where $C_{Q,\G} = \Gamma \backslash C_Q$ and $\G_{Q^\pr} = \G \cap Q^\pr(\Q)$. If $\G$ is neat, then $\G_{Q^\pr}$ is neat for all $Q^\pr$. The space $\G \backslash \ol{X}$ is a compact manifold with corners, which in particular implies that it is homeomorphic to a manifold with boundary \cite[Appendix]{bs}, so the results of \S \ref{top preliminaries} apply to it. 

\subsection{The vanishing conjecture for completed cohomology}\label{subsec: vanishing}

In this subsection we assume that $G$ is reductive. Fix a compact open subgroup $K^p \sub G(\A_f^p)$. Let $R$ be an adic ring with finitely generated ideal of definition $I$. We define completed cohomology of $G$ (with respect to $K^p$) to be
\[
\tH^\ast_?(K^p,R) := \tH^\ast_?\left( X_{K^p}, R \right) = \varprojlim_n \varinjlim_{K_p} H^\ast_?\left(  X_{K^p K_p}, R/I^n \right),
\]
where $? \in \{\emptyset,c\}$ and $K_p$ runs through the compact open subgroups of $G(\Q_p)$. We recall the quantities
\[
l_0 = l_0(G) := \mathrm{rank}(G(\R)) - \mathrm{rank}(A(\R)K_\infty)
\]
and
\[
q_0 = q_0(G) := \frac{\dim_\R X - l_0}{2},
\]
where $\mathrm{rank}$ denotes the rank as a Lie group. With these preparations, we may state the main vanishing conjecture of Calegari--Emerton:

\begin{conj}\label{ce vanishing}
Let $?\in \{\emptyset,c\}$. Then $\tH^n_?(K^p,\Z_p) = 0$ for all $n> q_0$.
\end{conj}

\begin{rema} 
While Conjecture \ref{ce vanishing} is not explicitly stated in \cite{ce}, it is a direct consequence of \cite[Conjecture 1.5(5)-(8) and Theorem 1.1(3)]{ce}. We will discuss \cite[Conjecture 1.5]{ce} in \S \ref{codimension}.
\end{rema}

We will focus on the following equivalent version, which is also implicit in \cite{ce}.

\begin{conj}\label{ce vanishing 2}
Let $?\in \{\emptyset,c\}$. Then $\tH^n_?(K^p,\Fp) = 0$ for all $n> q_0$.
\end{conj}

\begin{prop}
Conjecture \ref{ce vanishing} is equivalent to Conjecture \ref{ce vanishing 2}.
\end{prop}

\begin{proof}
That Conjecture \ref{ce vanishing} implies Conjecture \ref{ce vanishing 2} follows from \cite[Theorem 1.16(5)]{ce}. For the converse, note first that we have long exact sequences
\[
\dots \to \tH_?^i(K^p, \Z/p^{r-1}) \to \tH_?^i(K^p, \Z/p^r) \to \tH_?^i(K^p,\Fp) \to \dots 
\]
coming from the the corresponding long exact sequences at finite level, so by induction on $r$ we see that Conjecture \ref{ce vanishing 2} implies that $\tH_?^i(K^p, \Z/p^r) = 0$ for all $r$ and $n>q_0$. Conjecture \ref{ce vanishing} then follows since $\tH_?^i(K^p, \Z_p) = \varprojlim_r \tH_?^i(K^p, \Z/p^r)$.
\end{proof}

As usual in the Langlands program, adelic double quotients have the advantage that they make the Hecke actions and group actions transparent. These actions will, however, play essentially no role in this paper, and we found it simpler to work non-adelically. The rest of this subsection will discuss a version of Conjecture \ref{ce vanishing 2} in this language that we will treat. Let $\G \sub G(\Q)$ be an arbitrary arithmetic subgroup and set
\[
C_p = C_p(\G) := \{ \G \cap G(\A_f^p)K_p \mid K_p \sub G(\Q_p) \text{ compact open} \}.
\]
Informally, this is the set of arithmetic subgroups of $\G$ where we shrink the level at $p$. Armed with this definition, we set
\[
\wh{X} = \wh{X}(\G) = \wh{X}^G(\G) := \varprojlim_{\G^\prime \in C_p} \G^\prime \backslash X.
\]
We let $G(\R)_+$ denote the preimage of $G^{ad}(\R)^+$ under the natural map $G(\R) \to G^{ad}(\R)$; we also set $G(\Q)_+ = \G(\Q) \cap G(\R)_+$. We can then state the following conjecture.

\begin{conj}\label{connected vanishing}
Let $? \in \{\emptyset,c\}$ and assume that $\G \sub G(\Q)_+$ is an arithmetic subgroup. Then we have $\tH_?^n(\wh{X},\Fp) = 0$ for all $n>q_0$.
\end{conj}

This is the conjecture that we will focus on. The restriction to $\G \sub G(\Q)_+$ seems unnatural but this condition will feature in all our unconditional theorems, so we have included for the sake of easy referencing. A priori, (the natural generalization of ) Conjecture \ref{connected vanishing} is slightly stronger than Conjecture \ref{ce vanishing 2} because we allow arbitrary arithmetic subgroups as our ``base level'' instead of just congruence subgroups inside the identity component $G(\R)^+$ of $G(\R)$. Let us give a general discussion of the passage between disconnected spaces and their components, and formalize the implication relevant to this paper. To simplify notation, we drop the notation $\wt{M}$ used in \S \ref{top preliminaries} to denote the local system associated with a representation $M$, simply writing $M$ for the local system as well in the rest of this paper. 

\medskip

First, for any compact subgroup $K\sub G(\A_f)$, define
\[
\mf{X}_K := \mf{X}^G_K := G(\Q)^+ \backslash X \times G(\A_f) / K,
\]
where we now give $G(\A_f)$ the \emph{discrete} topology. Note that $\mf{X}_K = X_K$ when $K$ is open. In general, $\mf{X}_K$ is a manifold when $K$ is neat. If $K_1 \sub K_2$ are neat, with $K_1$ normal in $K_2$, then $K_2/K_1$ acts freely on $\mf{X}_{K_1}$ with quotient $\mf{X}_{K_2}$. We similarly define $\ol{\mf{X}}_{K}$, replacing $X$ by $\ol{X}$. In particular, using $\mf{X}_{K^p}$ and $\ol{\mf{X}}_{K^p}$, we may apply Theorem \ref{cc at finite level} to deduce that
\[
\wt{H}^i_?(K^p, \Fp) \cong H^i_?(X_K, \Map_{cts}(K_p, \Fp))
\]
where $K=K^p K_p$ with $K_p$ neat. Using the decomposition into connected components, we see that
\[
\wt{H}^i_?(K^p, \Fp) \cong \bigoplus_{g \in G(\Q)^+ \backslash G(\A_f) /K } H^i_?(\G_g \backslash X, \Map_{cts}(K_p, \Fp)).
\]
Here, the right $K_p$-module $\Map_{cts}(K_p, \Fp)$ (action via left translation) becomes a right $\G_g = G(\Q)^+ \cap gKg^{-1}$-module via the composition $\Gamma_g \to K \to K_p$ where the first map is conjugation by $g^{-1}$ and the second is the projection, and then a left $\G_g$-module by inversion. In particular, we have an isomorphism
\[
\Map_{cts}(K_p, \Fp) \cong \Map_{cts}(g_p K_p g_p^{-1}, \Fp)
\]
of $\G_g$-modules (with the obvious $\G_g$-structure on the right hand side). Then, note that the left $\G_g$-module $\Map_{cts}(g_p K_p g_p^{-1}, \Fp)$, where the action is via inverting the left translation action, is isomorphic to the left $\G_g$-module $\Map_{cts}(g_p K_p g_p^{-1}, \Fp)$ where the action is the right translation action (the isomorphism is given by inversion on $g_p K_p g_p^{-1}$). This proves the following:

\begin{prop}\label{passage to components 1}
Fix $i$ and $K^p$. Choose $K_p$ sufficiently small to make $K=K^p K_p$ neat. For any other $K^\prime \sub G(\A_f)$ compact open, set $\G^\prime = G(\Q)^+ \cap K^\prime$. Then $\wt{H}^i_?(K^p, \Fp) = 0$ if and only if $H^i_?(\G^\prime \backslash X, \Map_{cts}(K^\prime_p, \Fp)) = 0$ for all conjugates $K^\prime$ of $K$ in $G(\A_f)$, where $\G^\prime$ acts on $\Map_{cts}(K^\prime_p, \Fp))$ either via right translation or by inverting the left translation action.
\end{prop}

As a corollary we get the implication between Conjecture \ref{connected vanishing} and Conjecture \ref{ce vanishing 2}.

\begin{prop}\label{passage to components 2}
Let $? \in \{\emptyset, c\}$. Then Conjecture \ref{connected vanishing} for $?$ implies Conjecture \ref{ce vanishing 2} for $?$.
\end{prop}

\begin{proof}
Let $K^p \sub G(\A_f^p)$ be compact open and let $n > q_0$.  We first show that $H^n_?(\G \backslash X, \Map_{cts}(K_p, \Fp)) = 0$ for some (any) sufficiently small $K_p$, where $\G = G(\Q)^+ \cap K^p K_p $ acts on $\Map_{cts}(K_p, \Fp))$ via right translation. Consider $\wh{X} = \wh{X}(\G)$. By Conjecture \ref{connected vanishing}, $\wt{H}_?^n(\wh{X},\Fp) = 0$. By Theorem \ref{cc at finite level}, 
\[
 H^n_?(\G \backslash X, \Map_{cts}(H,\Fp)) = \wt{H}_?^n(\wh{X},\Fp) = 0
\]
where $H$ is the closure of $\G$ in $K_p$ and $\G$ acts on $\Map_{cts}(H,\Fp)$ via right translation. An application of Lemma \ref{big and small local systems} then gives that $H^n_?(\G \backslash X, \Map_{cts}(K_p, \Fp)) = 0$, as desired. Now the argument applies equally well when replacing $K^pK_p$ by any conjugate of it inside $G(\A_f)$, so by Proposition \ref{passage to components 1} we deduce that $\wt{H}^i_?(K^p, \Fp) = 0$.
\end{proof}

\subsection{The case of Hermitian symmetric domains}\label{subsec: hermitian}

In this subsection, we assume that $G$ is semisimple and that $X$ is a Hermitian symmetric domain. In this case, $l_0=0$ and $q_0 = (\dim_\R X)/2 = \dim_\C X$; we will simply write $d$ for this quantity. We briefly recall some material from the theory of hermitian symmetric domains and their boundary components; some references for this material are \cite{amrt,bb,he}. We do not assume that $G$ has no $\R$-anisotropic $\Q$-simple factors. 

\medskip

First, let us recall that an element $g\in G(\R)$ acts holomorphically on $X$ if and only if $g\in G(\R)_+$; see \cite[Proposition 11.3]{bb} (note that $G$ is assumed to be adjoint in this reference). The space $X=X^G$ has a bordification $X^\ast=X^{G,\ast}$ obtained by adding the rational boundary components of $X$, see \cite{bb}. To describe it, we make a definition. If $G$ is $\Q$-simple, let us call a parabolic subgroup $Q$ \emph{admissible} if there is no parabolic subgroup $Q^\pr$ with $Q \subsetneq Q^\pr \subsetneq G$. For general $G$, we will call a parabolic subgroup $Q$ admissible if its projection to every $\Q$-simple factor is admissible in the previous sense. Let $Q$ be such an admissible parabolic subgroup of $G$; we write $N_Q$ for its unipotent radical and $M_Q$ for its Levi quotient. $M_Q$ decomposes into an almost direct product $M_Q = M_{Q,\ell}M_{Q,h}$; see \cite[Item (5), p. 142]{amrt} (in the notation of that reference, we take $M_{Q,\ell}=\mathscr{G}_\ell$ and $M_{Q,h}=\mathscr{G}_h \cdot \mathscr{M}$). $M_{Q,\ell}$ is called the \emph{linear} part; it is a connected reductive group. $M_{Q,h}$ is called the \emph{Hermitian} part and it is a semisimple group whose symmetric space is a Hermitian symmetric domain. Our main result in the topological part of this paper is the following.

\begin{theo}\label{main top}
With assumptions as above, assume that Conjecture \ref{connected vanishing} holds for $M_{Q,h}$ for all admissible parabolics $Q$ of $G$ (including $Q=G$) and $?=c$. Then Conjecture \ref{connected vanishing} holds for $G$ and $?= \emptyset$. 
\end{theo}

The proof will occupy the rest of this subsection. Let us now describe the bordification $X^\ast$. Set-theoretically, 
\[
X^\ast = \bigsqcup_{Q \text{ admissible}} X^{M_{Q,h}} = \bigsqcup_{Q \supseteq P \text{ admissible}} X^{G,M_{Q,h}},
\]
where $X^{G,M_{Q,h}} := \bigsqcup_{Q^\pr \in C_Q} X^{M_{Q^\pr,h}}$ and we recall that $P$ is a fixed choice of a minimal parabolic subgroup. The action of $G(\Q)$ on $X$ extends to an action on $X^\ast$, but torsion-free arithmetic subgroups will no longer act freely (in general). The spaces $X^{G,M_{Q,h}}$ are stable under $G(\Q)$. If $\G \sub G(\Q)_+$ is a torsion-free arithmetic subgroup, then $\G \backslash X^\ast $ has a canonical structure of a projective algebraic variety over $\C$. Let us now assume that $\G$ is in addition neat, and let $\G_{M_{Q^\pr,h}}$ be the image of $\G_{Q^\pr}$ in $M_{Q^\pr,h}(\Q)$; this is a neat arithmetic subgroup. We have a stratification
\[
\G \backslash X^\ast = \bigsqcup_{Q \supseteq P \text{ admissible}} \bigsqcup_{Q^\pr \in C_{Q,\Gamma}} \G_{M_{Q^\pr,h}} \backslash X^{M_{Q^\pr,h}}
\]
of the quotient. By construction $\G_{M_{Q^\pr,h}}$ acts holomorphically on $X^{M_{Q^\pr,h}}$, so $\G_{M_{Q^\pr,h}} \sub M_{Q^\pr,h}(\Q)_+$.

\medskip

In \cite{zu}, Zucker constructs a $G(\Q)$-equivariant continuous map $\pi : \ol{X} \to X^*$ that we will make use of.\footnote{It is, strictly speaking, not necessary for us to use minimal compactifcations and Zucker's work \cite{zu}, as all we need is the resulting stratification of the Borel--Serre compactification which one may describe directly. Nevertheless, we have opted to include the minimal compactification in our discussion as it gives a conceptual way of understanding the stratification that we use, and why we use it.} With $Q$ as above, let us write $Y(Q) = \pi^{-1}(X^{M_{Q,h}})$. By \cite[(3.8), Proposition]{zu}, we have a natural homeomorphism
\[
Y(Q) \cong X^{M_{Q,h}} \times \ol{X}^{M_{Q,\ell}} \times X^{N_Q}
\]
and the projection maps
\[
Y(Q) \to Y(M_Q) := X^{M_{Q,h}} \times \ol{X}^{M_{Q,\ell}} \to \ol{X}^{M_{Q,\ell}}
\]
are $Q(\Q)$-equivariant (and fibre bundles). Write $L_Q = M_{Q,\ell}/(M_{Q,\ell} \cap M_{Q,h})$; the natural map $M_{Q,\ell} \to L_Q$ is a central isogeny and $\ol{X}^{M_{Q,\ell}} = \ol{X}^{L_Q}$. Then we remark that, in the displayed equation above, $Q(\Q)$ acts via the projection map $Q(\Q) \to M(\Q)$ on $Y(M_Q)$ and via the projection map $Q(\Q) \to L_Q(\Q)$ on $\ol{X}^{M_{Q,\ell}}$. In particular, we note that $Y(Q)$ is contractible and that if $\G$ is torsion-free, then $\G_Q$ acts freely on $Y(Q)$.

\medskip

We now begin the proof of Theorem \ref{main top}. Fix an arithmetic subgroup $\G \sub G(\Q)_+$. Our goal is to understand $\tH^\ast(\wh{X},\Fp)=\tH^\ast(\wh{\ol{X}},\Fp)$ in terms of the $\tH_c^\ast(\wh{X}^{M_{Q,h}},\Fp)$, where $\wh{X} =  \wh{X}^G(\G)$ and 
\[
\wh{\ol{X}} = \varprojlim_{\G^\prime \sub C_p(\G)}\G^\prime \backslash \ol{X}.
\]
By Proposition \ref{existence of neat subgroups} we may assume that $\G$ is neat without changing $\wh{X}$ and $\wh{\ol{X}}$. Let $S$ denote the closure of $\G$ in $G(\Qp)$. Proposition \ref{cc at finite level} then gives us the following description of $\tH^\ast(\wh{\ol{X}},\Fp)$.

\begin{prop}\label{step 1}
We have a canonical isomorphism 
\[
\tH^\ast(\wh{\ol{X}},\Fp) \cong H^\ast(\G \backslash \ol{X}, \Map_{cts}(S, \Fp)).
\]
\end{prop}

The ``stratification'' $(Y(Q))_{Q}$ of $\ol{X}$ induces a finite stratification $(\G_Q \backslash Y(Q))_Q$ of $\G \backslash \ol{X}$ into locally closed subsets, parametrized by $\G$-conjugacy classes of admissible parabolic subgroups $Q$. By repeated use of the excision sequence, it suffices for us to prove that
\[
H_c^i(\G_Q \backslash Y(Q), \Map_{cts}(S, \Fp) ) = 0 
\]
for $i>d$ and for all $Q$. From now on we fix $Q$ and drop the subscripts $-_Q$ from all associated algebraic groups for simplicity. Consider the proper map $f : \G_Q \backslash Y(Q) \to \G_M \backslash Y(M)$, which is a fibre bundle with fibre $\G_N \backslash X^{N_Q}$. Here $\G_N = N(\Q) \cap \G_Q$ and $\G_M$ is the image of $\G_Q$ under $Q(\Q) \to M(\Q)$. Set $S_N = S \cap N(\Qp)$; by strong approximation this is the closure of $\G_N$ in $N(\Qp)$ (and hence open). Then we have
\[
H_c^\ast(\G_Q \backslash Y(Q), \Map_{cts}(S, \Fp) ) = H^\ast_c(\G_M \backslash Y(M), Rf_\ast \Map_{cts}(S, \Fp)).
\]
Since $\G \cap S_N = \G_N$, $\G_M=\G_Q/\G_N$ acts by right translation on $S/S_N$.

\begin{prop}\label{step 2}
$f_\ast \Map_{cts}(S, \Fp) = \Map_{cts}(S/S_N, \Fp)$ with $\G_M$ acting by right translation, and $R^i f_\ast \Map_{cts}(S, \Fp) = 0$ for all $i \geq 1$.
\end{prop}

\begin{proof}
By Corollary \ref{formula for pushforward}, $R^i f_\ast \Map_{cts}(S, \Fp)$ is the local system on $\G_M \backslash Y(M)$ corresponding to the $\G_M$-representation $H^i (\G_N,\Map_{cts}(S,\Fp))$. When $i=0$, the description is clear since $\G_N$ is dense in $S_N$. In general, choose a continuous section $S \to S_N$ of the inclusion, which gives a homeomorphism $S \cong S/S_N \times S_N$ of right $S_N$-spaces. Arguing as in Lemma \ref{big and small local systems}, we see that
\[
H^i(\G_N, \Map_{cts}(S,\Fp)) \cong \Map_{cts}(S/S_N, \Fp) \otimes_{\Fp} H^i(\G_N, \Map_{cts}(S_N,\Fp)).
\]
By Proposition \ref{injectives are acyclic} and the injectivity of $\Map_{cts}(S_N,\Fp)$ (discussed in \S \ref{coh of unipotent groups}), the right hand side is $0$ when $i\geq 1$.
\end{proof}

So, we are down to computing $H^\ast_c(\G_M \backslash Y(M),\Map_{cts}(S/S_N, \Fp))$, for which we use the fibre bundle
\[
g : \G_M \backslash Y(M) \to \G_L \backslash \ol{X}^L,
\]
with fibre $\G_h \backslash X^{M_h}$. Here $\G_h = M_h(\Q) \cap \G_M$ and $\G_L = r(\G_M)$, where $ r : M \to L $ denotes the canonical map. We remark that $\G_h$ acts holomorphically on $X^{M_h}$, and hence $\G_h \sub M_h(\Q)_+$. The Leray spectral sequence reads
\[
H^i(\G_L \backslash \ol{X}^L, R^j g_! \Map_{cts}(S/S_N,\Fp)) \implies H^{i+j}_c(\G_M \backslash Y(M),\Map_{cts}(S/S_N, \Fp)).
\]
The key is then the following.

\begin{prop}\label{step 3}
$R^j g_! \Map_{cts}(S/S_N,\Fp))$ is a local system on $\G_L \backslash \ol{X}^L$ and vanishes for $j > \dim_\C X^{M_h}$.
\end{prop}

\begin{proof}
$R^j g_! \Map_{cts}(S/S_N,\Fp))$ is a local system with fibre $
H_c^j(\G_h \backslash X^{M,h}, \Map_{cts}(S/S_N,\Fp))$ by Proposition \ref{formula for proper pushforward}. Consider the closure $T_h$ of $\G_h$ in $M_h(\Qp)$, which we may also view as the closure of $\G_h$ in $S/S_N$. Write $S_h$ for the preimage of $T_h$ under $S \to S/S_N$. $S_h$ is a group containing $S_N$ as a normal subgroup, and $T_h=S_h/S_N$. Applying Lemma \ref{big and small local systems} with $G=S$, $H=S_h$, $K=S_N$ and $\G = \G_h$, $H_c^j(\G_h \backslash X^{M,h}, \Map_{cts}(S/S_N,\Fp))$ vanishes if 
\[
H_c^j(\G_h \backslash X^{M_h}, \Map_{cts}(T_h,\Fp)).
\]
But $H_c^j(\G_h \backslash X^{M,h}, \Map_{cts}(T_h,\Fp))$ is compactly supported completed $\Fp$-cohomology for $M_h$ by Proposition \ref{cc at finite level}, so this vanishes for $j > \dim_\C X^{M_h}$ by assumption.
\end{proof}

Before we put everything together, we need to relate $d$ to $\dim_\C X^{M_h}$ and $\dim_\R X^L$. Recall that $A_L$ is the maximal $\Q$-split torus in the center of $L$, and write $Z(N)$ for the center of $N$. The result is then the following.

\begin{lemm}\label{numerics}
$ \dim_\C X^{M_h} + \dim_\R X^L = d - \frac{1}{2}\left( \dim N - \dim Z(N) \right) - \dim A_L$.
\end{lemm}

\begin{proof}
The symmetric space $X^G$ has a decomposition
\[
X^G \cong X^{M_h} \times C(L) \times N(\R)
\]
as real manifolds\footnote{This is written as $D\cong F \times C(F) \times W(F)$ in \cite{amrt}; with respect to our notation $D=X^G$, $F=X^{M_h}$, $C(F)=C(L)$ and $W(F) = N(\R)$.} by \cite[Equation (4.1)]{amrt}. This gives
\[
\dim_\C X^{M_h} =  d - \frac{1}{2}\left( \dim_\R C(L) + \dim N \right).
\]
The space $C(L)$, called $C(F)$ in \cite{amrt}, is an open subset of $Z(N)(\R)$ and diffeomorphic to $L(\R)/K_{L,\infty}$ by \cite[Theorem 4.1(2)]{amrt}, where $K_{L,\infty}$ denotes a maximal compact subgroup of $L(\R)$. Thus $\dim_\R X^L = \dim_\R C(L) - \dim A_L $ and $\dim_\R C(L) = \dim Z(N)$. Combining this with the displayed equation above and rearranging gives the desired result.
\end{proof}

We may now put everything together to prove a more precise version of Theorem \ref{main top}. From now on we let $Q$ denote an arbitrary admissible parabolic of $G$ again, and set 
\[
\gamma(Q) =  \frac{1}{2}\left( \dim N_Q - \dim Z(N_Q) \right) + \dim A_{L_Q} + \text{ss.rank}_\Q(L_Q)
\]
whenver $Q\neq G$. Here $\text{ss.rank}_\Q(H)$, for $H$ a reductive group over $\Q$, denotes the $\Q$-rank of the derived group of $H$ (the `semisimple $\Q$-rank' of $H$). Note that $\gamma(Q)$ is non-negative and only depends on the conjugacy class of $Q$. In fact, $\dim A_{L_Q}$, and hence $\gamma(Q)$, is always positive. This follows, for example, from \cite[\S 4.2, Equation (2)]{bs}, upon noting that $\dim A_{L_Q}=\dim A_Q$. More precisely, this shows that $\dim A_{L_Q}$ is equal to the number of $\Q$-simple adjoint factors of $G \to H$ in which the projection of $Q$ is not equal to $H$.

\begin{theo}\label{main top precise}
Assume that Conjecture \ref{connected vanishing} holds for $M_{Q,h}$ for all admissible parabolics $Q$ of $G$ (including $Q=G$) and $?=c$. Then the natural map
\[
H_c^i (\wh{X}, \Fp) \to H^i(\wh{X}, \Fp)
\]
is an isomorphism when $i > d+1 - \inf_{Q\neq G} \gamma(Q)$, and surjective for $i = d+1 - \inf_{Q\neq G} \gamma(Q)$. In particular, Conjecture \ref{connected vanishing} holds for $G$ and $?= \emptyset$, and $H_c^d (\wh{X}, \Fp) \to H^d(\wh{X}, \Fp)$ is surjective.
\end{theo}

\begin{proof}
This merely summarizes the work done above, so we will be rather brief. By Proposition \ref{step 1} and repeated use of the excision sequence, it suffices to show that, for all $Q\neq G$, $H_c^i(\G_Q \backslash Y(Q), \Map_{cts}(S, \Fp) ) = 0 $ for $i > d - \gamma(Q)$. Propositions \ref{step 2} and \ref{step 3} then give us a spectral sequence 
\[
H^j(\Gamma_L \backslash \ol{X}^L, R^k g_!\Map_{cts}(S/S_N,\Fp)) \implies H_c^{j+k}(\G_Q \backslash Y(Q), \Map_{cts}(S, \Fp) ) 
\]
and shows that $R^k g_!\Map_{cts}(S/S_N,\Fp)$ is a local system which is $0$ for $k > \dim_\C X^{M_h}$. By \cite[Corollary 11.4.3]{bs} the cohomology of local systems on $\G_L \backslash \ol{X}^L$ vanishes in degrees $> \dim_{\R}X^L - \text{ss.rank}_\Q(L)$, so we see that $H_c^i(\G_Q \backslash Y(Q), \Map_{cts}(S, \Fp) ) = 0$ for $i > \dim_{\C}X^{M_h} + \dim_{\R}X^L - \text{ss.rank}_\Q(L)$. Finally, by Lemma \ref{numerics}, this quantity is equal to $d - \gamma(Q)$ as desired, finishing the proof.
\end{proof}

\begin{rema}
When $\inf_{Q \neq G} \gamma(Q) \geq 2$, $H_c^d (\wh{X}, \Fp) \to H^d(\wh{X}, \Fp)$ is an isomorphism. This is typically the case, and it is relatively straightforward to check in any given case. However, there are examples where $\inf_{Q \neq G} \gamma(Q) = 1$, such as $G= \SL_{2/F}$ with $F$ totally real. In the example $G=\SL_{2/F}$, note that the map $H_c^d (\wh{X}, \Fp) \to H^d(\wh{X}, \Fp)$ is not an isomorphism when $F=\Q$. When $F\neq \Q$ the map should be an isomorphism (which can be proved by using Waldschmidt's bound on the Leopoldt defect in the proof above instead of \cite[Corollary 11.4.3]{bs}).
\end{rema}

\begin{rema}\label{remark on stratifications}
The reader familiar with the Borel--Serre compactification may wonder if one could not have used the ``usual'' stratification, indexed by all rational parabolic subgroups of $G$. This is possible, but it simply amounts to a more complicated version of the analysis above. Let us explain this briefly. The strata in the usual stratification are locally symmetric spaces for rational parabolics $P$ of $G$, and we would want to prove vanishing results for compactly supported cohomology on these strata of the same local system as above. If $M$ is the Levi quotient of $P$, the analogue of Proposition \ref{step 2} goes through in the same way and essentially reduces us to compactly supported completed cohomology of $M$. No trivial bound will be sufficient, and in general $X^M$ won't be Hermitian, so we are forced to try to find an almost direct factor of $M$ whose symmetric space is Hermitian (just like for admissible $P$) to get a better bound. It is not so hard to see (e.g. by looking at Dynkin diagrams) that there is a canonical admissible $Q$ whose Hermitian part is an almost direct factor of $M$, and using this one can push through the analysis. The strata of the coarser stratification that we use are simply the unions of the $X^P$ for all $P$ which have the same Hermitian part, i.e. which are associated with the same admissible $Q$ by the procedure above, and they fit together in such  way that it is much better to analyze them together rather than separately.
\end{rema}

\subsection{The Calegari--Emerton conjectures on completed homology}\label{codimension}

We return to the setting of \S \ref{subsec: vanishing}. We recall from \cite{ce} that completed homology of $G$ with tame level $K^p \sub G(\A_f^p)$ values in an adic ring $R$ is defined as
\[
\tH_i(K^p,R) := \varprojlim_{K_p} H_i(X_{K^p K_p},R),
\]
where $K_p$ runs through the compact open subgroups of $G(\Qp)$. One may define completed Borel--Moore homology $\tH_i^{BM}(K^p,R)$ similarly (again see \cite{ce}). Let $? \in \{\emptyset, BM\}$. For any compact open subgroup $K_p \sub G(\Qp)$, $\tH_i^{?}(K^p,\Zp)$ is a finitely generated right module for the Iwasawa algebra $\Zp \llbracket K_p \rrbracket $, which is an Auslander--Gorenstein ring and has well-defined codimension (or grade) function on its finitely generated right modules, defined by 
\[
cd(M) = \inf \{j \mid \Ext^j_{\Zp \llbracket K_p \rrbracket }(M, \Zp \llbracket K_p \rrbracket ) \neq 0 \}.
\]
We refer to \cite[\S 2.5]{aw} for more details on the properties of the codimension function. In particular we remark that by general properties, $cd(\tH_i^{?}(K^p,\Zp))$ is independent of the choice of $K_p$. Recall the quantities $q_0$ and $l_0$ from \S \ref{subsec: vanishing}. We may then state a slightly weaker version of \cite[Conjecture 1.5]{ce}. For simplicity, from now on we write $\tH_i^?$ for $\tH_i^{?}(K^p,\Zp)$.

\begin{conj}[Calegari--Emerton]\label{ce vanishing homology}
Let $? \in \{ \emptyset, BM \}$. Then the following holds:
\begin{enumerate}
\item If $i < q_0$, then $cd(\tH_i^?) \geq q_0 + l_0 -i$.

\smallskip

\item $\tH_{q_0}^?$ has codimension $l_0$.

\smallskip

\item $\tH_{q_0}^?$ is $p$-torsionfree.

\smallskip

\item $\tH_i^? = 0$ for $i > q_0$.
\end{enumerate}
\end{conj}

The difference between this conjecture and \cite[Conjecture 1.5]{ce} is that the latter predicts $cd(\tH_i^?) > q_0 + l_0 -i$ when $i < q_0$. Completed (Borel--Moore) homology is closely related to completed (compacty supported) cohomology via \cite[Theorem 1.1]{ce}. Moreover, completed homology and completed Borel--Moore homology are related via the two Poincar\'e duality spectral sequences
\[
E_2^{ij} = \Ext_A^i(\tH_j,A) \implies \tH_{D-i-j}^{BM};
\]
\[
E_2^{ij} = \Ext_A^i(\tH_j^{BM},A) \implies \tH_{D-i-j},
\]
where $A = \Zp \llbracket K_p \rrbracket$ and $D= \dim_{\R} X=2q_0+l_0$; see \cite[\S 1.3]{ce}. We have the following relation between Conjecture \ref{ce vanishing} and Conjecture \ref{ce vanishing homology}.

\begin{prop}\label{implications}
Conjecture \ref{ce vanishing} for compactly supported completed cohomology implies Conjecture \ref{ce vanishing homology}(3)-(4) for completed Borel--Moore homology and Conjecture \ref{ce vanishing homology}(1) for completed homology. Similarly, Conjecture \ref{ce vanishing} for completed cohomology implies Conjecture \ref{ce vanishing homology}(3)-(4) for completed homology and Conjecture \ref{ce vanishing homology}(1) for completed Borel--Moore homology. 
\end{prop}

\begin{proof}
The first part is essentially \cite[Corollary 4.2.3]{scho}; the proof there works verbatim (note that there is a small typo in that proof; the quantity $c$ there should be chosen to be minimal, not maximal, with respect to the given property). For the second part the proof is the same, swapping the roles of completed cohomology and compactly supported completed cohomology, and completed homology and completed Borel--Moore homology.
\end{proof}

Let us also indicate that Conjecture \ref{ce vanishing homology}(2) is known for completed homology when $G$ is semisimple and $l_0=0$; this is part of \cite[Theorem 1.4]{ce} (and follows from \cite{ce-annals} and known limit multiplicity formulas for discrete series). More precisely, let  $\G \sub G(\Q)$ is an arithmetic subgroup with closure $\ol{\G} \sub G(\Qp)$ and let $\wh{X} = \wh{X}(\G)$. Then $\tH^{q_0}(\wh{X},\Zp)[1/p]$ is an admissible $\Qp$-Banach space representation of $\ol{\G}$ of corank $0$ by the results of \cite{ce-annals}. Dualizing, we see that the completed homology space 
\[
\tH_{q_0}(\wh{X},\Zp) := \varprojlim_{\G^\prime \sub C_p(\G)} H_{q_0}(\G^\prime \backslash X, \Zp)
\]
has codimension $0$ as a $\Zp \llbracket \ol{\G} \rrbracket$-module. If $A \to B$ is a (left and right) flat map of (left and right) Noetherian rings and $M$ is a finitely generated right $A$-module, then one easily sees that $\Ext^i_B(M\otimes_A B,B) \cong B \otimes_A \Ext^i_A(M,A)$. In particular, if $A\to B$ is (left and right) faithfully flat, then $\Ext_A^i(M,A) = 0$ if and only if $\Ext_B^i(M\otimes_A B,B) = 0$. By an analysis of components similar to that preceding Proposition \ref{passage to components 1}, one sees that if $\G = G(\Q)^+ \cap K^p K_p$, then $\tH_{q_0}(\wh{X},\Zp)\otimes_{\Zp \llbracket \ol{\G} \rrbracket} \Zp \llbracket K_p \rrbracket$ is a direct summand of $\tH_{q_0}(K^p,\Zp)$, and hence the latter has codimension $0$ as a $\Zp \llbracket K_p \rrbracket$-module. Here we take $K_p$ to be sufficiently small so that it is neat and pro-$p$; then $\Zp \llbracket \ol{\G} \rrbracket \to \Zp \llbracket K_p \rrbracket$ is flat (indeed projective) by \cite[Lemma 4.5]{brumer}, and hence faithfully flat since $\Zp \llbracket \ol{\G} \rrbracket$ is then local. For ease of reference, let us state the result below.

\begin{theo}\label{codimension 0}
Assume that $G$ is semisimple and that $l_0 = 0$. Then the codimension of $\tH_{q_0}$ is equal to $0$.
\end{theo}

\section{Shimura varieties} \label{sec: shimura}

In this section we discuss Shimura varieties of Hodge and (pre-) abelian type, and how the conditional results of \S \ref{sec: cc} together with the results \S \ref{sec: perfectoid} give many unconditional cases of Conjectures \ref{ce vanishing} and \ref{ce vanishing homology}.

\subsection{Recollections on Shimura varieties}

We use the definition and conventions for Shimura data, morphisms of Shimura data, and connected Shimura data from \cite{de}; see also \cite{mi}. Given a Shimura datum $(G,X)$, there are three other data which one can attach to it, one Shimura datum and two connected Shimura data. They are as follows
\begin{itemize}
\item The connected Shimura datum $(G^{der},X^+)$;

\item The connected Shimura datum $(G^{ad},X^+)$;

\item The Shimura datum $(G^{ad},X^{ad})$.
\end{itemize}

Here $X^+\sub X$ is any choice of a connected component, and if $h\in X$, then $X^{ad}$ is the $G^{ad}(\R)$-conjugacy class of the composition of $h$ with $G_{\R} \to G^{ad}_{\R}$ (this is independent of the choice of $h$). The Shimura datum $(G^{ad},X^{ad})$ will only feature when we discuss the Hodge--Tate period map later, the other two will feature throughout the rest of this article. We recall that a Shimura datum $(G,X)$ is said to be of Hodge type if there exists a Siegel Shimura datum $(G^\prime, X^\prime)$ and a closed immersion $(G,X) \to (G^\prime,X^\prime)$ of Shimura data. A Shimura datum $(G,X)$ is said to be of abelian type if there exists a Shimura datum $(G_1,X_1)$ of Hodge type and a central isogeny $G_1^{der} \to G^{der}$ which induces an isomorphism $(G^{ad}_1,X_1^+) \cong (G^{ad},X^+)$. We make the following slightly more general definition, following \cite[2.10]{mo}.

\begin{defi}
Let $(G,X)$ be a connected Shimura datum. We say that $(G,X)$ is of pre-abelian type if there exists a Shimura datum $(\tG,\tX)$ of Hodge type such that $(G^{ad},X) \cong (\tG^{ad},\tX^+)$. We say that a Shimura datum $(G,X)$ is of pre-abelian type if $(G^{der},X^+)$ is of pre-abelian type.
\end{defi}

\begin{rema}
Recall that, if $G$ is semisimple, then by the convential definition $G$ admits a connected Shimura datum $(G,X)$ if and only if $G$ has no compact $\Q$-factors and $X^G$ is a hermitian symmetric domain; in this case $X\cong X^G$. The assumption that $G$ has no compact $\Q$-factors could be dropped, but we will keep phrasing our results in terms of Shimura data for simplicity.
\end{rema}

To be able to apply the inductive arguments from \S \ref{sec: cc}, we will need the following lemma.

\begin{lemm}\label{boundary components}
Assume that $G$ admits a connected Shimura datum of pre-abelian type and let $Q\sub G$ be an admissible parabolic with hermitian part $M_h$. Then $M_h$ admits a connected Shimura datum of pre-abelian type.
\end{lemm}

\begin{proof}
The assertion does not depend on the choice of $G$ inside the isogeny class of $G$, so we may assume that $(G,X)=(G_1^{der},X_1^+)$ with $(G_1,X_1)$ a Shimura datum of Hodge type. The assertion then follows from the well known fact that the rational boundary components of $(G_1,X_1)$ are of Hodge type. 
\end{proof}

\subsection{Results for semisimple groups}

The following is the main theorem of this paper on the Calegari--Emerton conjectures; at this point the proof is simply a summary of the results so far together with Corollary \ref{compactly supported vanishing}, which we prove using $p$-adic methods in the next section.

\begin{theo}\label{main semisimple}
Let $G$ be a semisimple group which admits a connected Shimura datum of pre-abelian type. Then Conjectures \ref{ce vanishing}, \ref{connected vanishing} and \ref{ce vanishing homology} hold for $G$. Moreover, for any $K^p$, the natural map $\wt{H}_c^i(K^p,\Zp) \to \wt{H}^i(K^p,\Zp)$ is an isomorphism for $i > d+1 - \inf_{Q\neq G} \gamma(Q)$ and surjective for $i = d+1 - \inf_{Q\neq G} \gamma(Q)$, where $d = \dim_{\C}X^G$, $Q$ is an admissible parabolic subgroup of $G$ and we recall that the quantities $\gamma(G)$ are defined in \S \ref{subsec: hermitian}. 
\end{theo}

\begin{proof}
We start with Conjecture $\ref{connected vanishing}$. For $?=c$, this Corollary \ref{compactly supported vanishing}. For $?=\emptyset$, it then follows from Lemma \ref{boundary components} and Theorem \ref{main top precise}. The more precise statement about the map $\wt{H}_c^i(K^p,\Zp) \to \wt{H}^i(K^p,\Zp)$ follows from Theorem \ref{main top precise}, Lemma \ref{big and small local systems} and an analysis of components as in the proof of Proposition \ref{passage to components 1}. Conjecture \ref{ce vanishing} then follows, and as does Conjecture \ref{ce vanishing homology} (using Proposition \ref{implications}, Theorem \ref{codimension 0} and the fact that $\wt{H}_d(K^p,\Zp) = \wt{H}_d^{BM}(K^p,\Zp)$).
\end{proof}

\subsection{Results for reductive groups}

Here we will briefly indicate what type of results can be proved towards the Calegari--Emerton conjectures for more general reductive groups. Recall that if $(G,X)$ is a Shimura datum, then $X^+$ need not equal the symmetric space $X^G$ in general. Indeed, $X^+ \cong G(\R)/Z(\R)K_\infty$, where $Z\sub G$ is the center and $K_\infty \sub G(\R)$ is a maximal compact subgroup. Recall that $A\sub Z$ is the maximal $\Q$-split subtorus and set
\[
Z^a = \bigcap_\chi \Ker \chi,
\]
where $\chi$ runs over the characters of $Z$ defined over $\Q$. Then $Z = Z^a A$ with $A\cap Z^a$ finite, and $X^G \to X^+$ is a (trivial) fibration with fiber $Z^a(\R)/(Z^a(\R) \cap K_\infty)$. In particular, $X^G \cong X^+$ if and only if $Z^a(\R)$ is compact. Note that this is equivalent to all arithmetic subgroups of $Z$ being finite, and to $l_0(Z)=0$. When this happens, we get clean results. Let $d=\dim_{\C}X$.

\begin{theo}\label{main reductive}
Assume that $G$ admits a Shimura datum of pre-abelian type and that $Z^a(\R)$ is compact. Then Conjectures \ref{ce vanishing}, \ref{connected vanishing} and \ref{ce vanishing homology} hold for $G$.
\end{theo}

\begin{proof}
We start with Conjecture \ref{connected vanishing}. Fix a neat arithmetic subgroup $\G \sub G(\Q)_+$ and $n > q_0$. Let $T= G/G^{der}$ be the cocenter of $G$. Since $Z \to T$ is an isogeny, all arithmetic subgroups of $T$ are finite as well. In particular, the image of $\G$ in $T(\Q)$ is neat, hence trivial. So $\G$ is contained in $G^{der}(\Q)_+$, and one readily sees that Conjecture \ref{connected vanishing} for $G$ is equivalent to Conjecture \ref{connected vanishing} for $G^{der}$, which follows from Theorem \ref{main semisimple}. Conjecture \ref{ce vanishing} then follows, and as does Conjecture \ref{ce vanishing homology} (using Proposition \ref{implications}) apart from part (2). For a proof of this we refer to Corollary \ref{general reductive coro} below, though we also note that one could give an easier proof in this special case. The last statement follows from the corresponding statement for $G^{der}$ by the same arguments as in Theorem \ref{main semisimple}. 
\end{proof}

\begin{rema}
We have elected to drop the statement that $\wt{H}_c^i(K^p,\Zp) \to \wt{H}^i(K^p,\Zp)$ is an isomorphism (or surjective) in a range of degrees, but of course the proof also shows that we get this. We will continue to drop this statement throughout this section.
\end{rema}

\begin{coro}\label{cor hodge}
Assume that $G$ admits a Shimura datum of Hodge type. Then Conjectures \ref{ce vanishing}, \ref{connected vanishing} and \ref{ce vanishing homology} hold.
\end{coro} 

\begin{proof}
If $G$ admits a Shimura datum of Hodge type, then $Z^a(\R)$ is compact, so Theorem \ref{main reductive} applies.
\end{proof}

When $Z^a(\R)$ is non-compact, the Leopoldt conjecture interferes in deducing the Calegari--Emerton conjectures for $G$ from $G^{der}$ or $G^{ad}$. Indeed, if $G=T$ is a torus, then the Leopoldt conjecture for $T$ is equivalent to Conjecture \ref{ce vanishing} for $T$; see \cite[\S 4.3.3]{hi} (note that Hill uses the symmetric spaces $G(\R)/K_\infty$ instead of our $X^G$). We recall this briefly (also recall that tori satisfy the congruence subgroup property). Let $K=K^p K_p$ be a compact open subgroup of $T(\A_f)$ with $K^p$ arbitrary and $K_p$ neat. Set $\G = T(\Q) \cap K$; this is a finitely generated torsion-free abelian group. Let $\wh{\G}$ be the $p$-adic completion of $\G$ and consider the natural map $f : \wh{\G} \to K_p$; set $\Delta = \Ker f$ and $I = \mathrm{Im}\,f $. $\Delta$ is a finite free $\Zp$-module and the Leopoldt conjecture asserts that $\Delta = 0$ (this assertion is independent of the choice of $K$). An application of \cite[Lemma 14]{hi} gives that
\[
H^i(\G, \Map_{cts}(I,\Fp)) = \Hom_{\Zp} ( \wedge^i_{\Zp} \Delta, \Fp),
\]
and, by Lemma \ref{big and small local systems}, $H^i(\G, \Map_{cts}(I,\Fp))$ vanishes simultaneously with $H^i(\G,\Map_{cts}(K_p,\Fp))$, so by Proposition \ref{passage to components 1} the vanishing of $\Delta$ is equivalent to Conjecture \ref{ce vanishing 2}. In fact, the Leopoldt conjecture is also equivalent to Conjecture \ref{ce vanishing homology}(2) for $T$ (note that $q_0(T)=0$). This is certainly also well known; we give a very brief sketch of the proof.

\begin{prop}\label{leopoldt and codimension}
Let $K^p \sub T(\A_f^p)$ be compact open. Then the codimension of $\wt{H}_0(K^p,\Zp)$ is $l_0 - \mathrm{rank}_{\Zp}\Delta$. In fact, the projective dimension of $\wt{H}_0(K^p,\Zp)$ is $l_0 - \mathrm{rank}_{\Zp}\Delta$.
\end{prop}

\begin{proof}
Choose $K_p$ neat and set $\G = T(\Q)^+ \cap K^p K_p$. As a right $K_p$-module, a straightforward computation (using the commutativity of $T$) shows that
\[
\wt{H}_0(K^p, \Zp) \cong \bigoplus_t \Zp \llbracket I \backslash K_p \rrbracket
\]
where $t$ runs over the finite set $T(\Q)^+ \backslash T(\A_f) / K^p K_p$ and $I$ denotes the closure of $\G$ in $K_p$. Set $M = \Zp \llbracket I \backslash K_p \rrbracket$, $A= \Zp \llbracket I  \rrbracket$ and $B = \Zp \llbracket K_p \rrbracket$; $B$ is a projective (left and right) $A$-module by \cite[Lemma 4.5]{brumer} and $M$ is a finitely generated right $B$-module, which is isomorphic to $\Zp \otimes_A B$. Then $\Ext^i_B(M,B) \cong B \otimes_A \Ext^i_A(\Zp, A) $, so the codimension of $M$ as a right $B$-module is equal to the codimension of $\Zp$ as a right $A$-module. Since $I \cong \Zp^{l_0 - \mathrm{rank}_{\Zp}\Delta}$, a computation using the Koszul complex shows that the codimension of $\Zp$ is $l_0 - \mathrm{rank}_{\Zp}\Delta$. This finishes the proof of the first part. For the second part about the projective dimension, note that the Koszul complex of $A$ is a resolution $P_\bu$ of $\Zp$ of length $l_0 - \mathrm{rank}_{\Zp}\Delta$ by finite free $A$-modules. It follows that $P_\bu \otimes_A B$ is a resolution of $M$ of length $l_0 - \mathrm{rank}_{\Zp}\Delta$ by finite free $B$-modules. Together with the first part, this finishes the proof of the second part.
\end{proof}

We may now give the most general result for reductive groups that we can prove.

\begin{theo}\label{general reductive}
Let $G$ be a connected reductive group over $\Q$ with center $Z$. Assume that the Leopoldt conjecture holds for $Z$ and that $G^{ad}$ admits a Shimura datum of abelian type. Then Conjecture \ref{connected vanishing} holds for $G$.
\end{theo}

\begin{proof}
Let $\G_0 \sub G(\Q)_+$ be an arithmetic subgroup and choose a sufficiently small neat $K_p$ which is a product $K_p = K_p^Z \times K_p^{ad}$ of a compact open $K_p^Z \sub Z(\Qp)$ and a compact open $K_p^{ad} \sub G^{der}(\Qp)$; note that the image of $K_p^{ad}$ in $G^{ad}(\Qp)$ is open and isomorphic to $K_p^{ad}$; we will conflate the two (this explains the notation). Set $\G = \G_0 \cap K_p$. Let $\G_Z = \G \cap Z(\Q)$ and consider the closure $\ol{\G_Z}$ of $\G_Z$ inside $K_p^Z$. It is not clear to us a priori if $\ol{\G}_Z$ is saturated inside $K_p^Z$ (as a $\Zp$-module), but if not we may achieve this by replacing $K_p^Z$ with a smaller subgroup without changing $\G_Z$, so we may assume this. This implies that the closure of the image of the projection $\G \to K_p^Z$ is equal to $\ol{\G}_Z$. To see this, let $C=G/G^{der}$ be the cocenter of $G$. The image $\G_C$ of $\G$ under the projection $\G \to C(\Qp)$ is an arithmetic subgroup inside $K_p^Z$ which contains $\G_Z$ as a finite index subgroup. It follows that the closure $\ol{\G}_C$ of $\G_C$ inside $K_p^Z$ contains $\ol{\G}_Z$ as a finite index subgroup, but since $\ol{\G}_Z$ is saturated they must be equal. From this, it follows that the composition
\[
\G \to K_p \to K_p/\ol{\G}_Z,
\]
where the first map is the inclusion and the second is the natural projection, is equal to the composition
\[
\G \to K_p^{ad} \to K_p^{ad} \times K_p^Z/\ol{\G}_Z = K_p / \ol{\G}_Z,
\]
where the first map is the projection and the second is the inclusion which is trivial on the second factor. We will use these facts in the calculation below.

\medskip

Now, by Proposition \ref{cc at finite level} and Lemma \ref{big and small local systems}, it suffices to show that 
\[
H_?^n(\G \backslash X, \Map_{cts}(K_p,\Fp)) = 0
\]
for $n > q_0 = q_0(G) = q_0(G^{ad})$. Let $\G_{ad}$ be the image of $\G$ in $G^{ad}(\Q)^+$. Consider the proper fibration $ \pi : \G \backslash X \to \G_{ad}\backslash X^{ad}$ with fiber $\G_Z \backslash X^Z$ (in this proof and the next only, $X^{ad} = X^{G^{ad}}$) and the corresponding Leray spectral sequence
\[
H^r_?(\G_{ad}\backslash X^{ad}, R^s\pi_\ast \Map_{cts}(K_p,\Fp)) \implies H_?^{r+s}(\G \backslash X, \Map_{cts}(K_p,\Fp)).
\]
By Proposition \ref{formula for proper pushforward}, $R^s\pi_\ast \Map_{cts}(K_p,\Fp)$ is the local system corresponding to $H^s(\G_Z, \Map_{cts}(K_p,\Fp))$.  Using the discussion on Leopoldt's conjecture above, the assumption that Leopoldt holds for $Z$, and Lemma \ref{big and small local systems}, we see that $H^s(\G_Z, \Map_{cts}(K_p,\Fp))=0$ for $s >0$. We then compute 
\[
\Map_{cts}(K_p,\Fp)^{\G_Z} \cong \Map_{cts}(K^Z_p/\ol{\G}_Z,\Fp) \otimes \Map_{cts}(K^{ad}_p,\Fp)
\]
as $\G_{ad}$-modules, where $\G_{ad}$ acts trivially on the first factor, which is an $\Fp$-vector space that we call $V$ (this uses the detailed setup above). So, the Leray spectral sequence reduces to
\[
H_?^n(\G \backslash X, \Map_{cts}(K_p,\Fp)) \cong H^n_?(\G_{ad}\backslash X^{ad}, V\otimes \Map_{cts}(K^{ad}_p,\Fp)) \cong H^n_?(\G_{ad}\backslash X^{ad}, \Map_{cts}(K^{ad}_p,\Fp)) \otimes V.
\]
By Lemma \ref{big and small local systems} and Theorem \ref{main semisimple}, $H^n_?(\G_{ad}\backslash X^{ad}, \Map_{cts}(K_p^{ad},\Fp))$ vanishes for $n>q_0$. This finishes the proof.
\end{proof}

\begin{coro}\label{general reductive coro}
Keep the notation and assumptions of Theorem \ref{general reductive}. Then Conjectures \ref{ce vanishing} and \ref{ce vanishing homology} hold for $G$.
\end{coro}

\begin{proof}
Note that $l_0 = l_0(G) = l_0(Z)$ and $q_0= q_0(G) = q_0(G^{ad})$. Fix $K^p\sub G(\A_f^p)$. Using Theorem \ref{general reductive}, everything apart from Conjecture \ref{ce vanishing homology}(2) follows as before, and additionally $\wt{H}_{q_0}(K^p,\Zp) = \wt{H}_{q_0}^{BM}(K^p,\Zp)$ . The argument in Proposition \ref{implications} also shows that $\wt{H}_{q_0}(K^p,\Zp)$ has codimension $\geq l_0$, so we need to show the opposite inequality. As in the proof of Theorem \ref{general reductive}, choose a neat $K_p \sub G(\Qp)$ which can be written as a product $K_p = K_p^{ad} \times K_p^Z$ with $K_p^Z \sub Z(\Qp)$ and $K_p^{ad} \sub G^{der}(\Qp)$, and set $\G = G(\Q)^+ \cap K^p K_p$ and $\G_Z = \G \cap Z(\Q)$; again we rig it so that the closure $\ol{\G}_Z$ is saturated inside $K_p^Z$. 

\medskip

We then have $\wt{H}_{q_0}(K^p,\Zp) \cong \Hom_{\Zp}(\wt{H}^{q_0}(K^p,\Zp), \Zp)$ by \cite[Theorem 1.1(3)]{ce} and the vanishing of $\wt{H}^{q_0+1}(K^p,\Zp)$, so it suffices to prove that $\wt{H}^{q_0}(K^p,\Zp)$ has a sub $K_p$-representation of injective dimension $\leq l_0$. Since $H^{q_0}(\G \backslash X, \Map_{cts}(K_p,\Zp))$ is a direct summand of $\wt{H}^{q_0}(K^p,\Zp)$, it suffices to show that $H^{q_0}(\G \backslash X, \Map_{cts}(K_p,\Zp))$ has a submodule of injective dimension $\leq l_0$. Here we view $\Map_{cts}(K_p,\Zp)$ as a left $\G$-module via inverting the left translation action; it has a commuting left $K_p$-action via right translation which gives $H^{q_0}(\G \backslash X, \Map_{cts}(K_p,\Zp))$ its structure of a left $K_p$-module. Using the computations in the proof of Theorem \ref{general reductive} with $\Fp$ replaced by $\Z/p^r$ and taking inverse limits over $r$ (which commute with cohomology by \cite[Proposition 1.2.12]{em}), we see that 
\[
H^{q_0}(\G \backslash X, \Map_{cts}(K_p,\Zp)) \cong H^{q_0}(\G_{ad}\backslash X^{ad}, \Map_{cts}(K^{ad}_p,\Zp)) \wh{\otimes}_{\Zp} \Map_{cts}(K_p^Z,\Zp)^{\G_Z},
\]
as left $K_p = K_p^{ad}\times K_p^Z$-representations. By Proposition \ref{leopoldt and codimension} and the assumption on $Z$, $\Map_{cts}(K_p^Z,\Zp)^{\G_Z}$ has injective dimension $l_0$ as a $K_p^Z$-representation. By Theorem \ref{codimension 0} and the discussion preceding it, $H^{q_0}(\G_{ad}\backslash X^{ad}, \Map_{cts}(K^{ad}_p,\Zp))$ contains an injective admissible $K_p^{ad}$-subrepresentation $W$. It follows that $W \wh{\otimes}_{\Zp} \Map_{cts}(K_p^Z,\Zp)^{\G_Z}$ is a sub $K_p$-representation of $H^{q_0}(\G \backslash X, \Map_{cts}(K_p,\Zp))$ of injective dimension $\leq l_0$, as desired.
\end{proof}

\begin{rema} We make a few additional remarks on these results.
\begin{enumerate}
\item Examples of cases when Theorem \ref{general reductive} and Corollary \ref{general reductive coro} are unconditional include  $G= \mathrm{Res}_{\Q}^F\GSp_{2g}$ for abelian totally real fields $F$, since the Leopoldt conjecture is known for tori which split over an abelian extension of $\Q$. One could also get weaker results with no condition on the center by assuming the known bounds for the Leopoldt defect.

\smallskip

\item Conjecture \ref{ce vanishing homology} has a natural analogue for $\Fp$-coefficients, stated in \cite[\S 1.7]{ce}. Our methods prove this conjecture too under the same assumptions. We content ourselves by noting that the arguments to prove Proposition \ref{implications} and Corollary \ref{general reductive coro} go through with only superficial changes for $\Fp$-coefficients (though one could simplify the argument in Corollary \ref{general reductive coro} for $\Fp$-coefficients). Note here that Theorem \ref{codimension 0} implies its $\Fp$-version when one knows $p$-torsionfreeness of $\wt{H}_{q_0}$, using the results of \cite[\S 1.7]{ce}.
\end{enumerate}

\end{rema}

\section{Perfectoid Shimura varieties}\label{sec: perfectoid}

\subsection{Preparations in $p$-adic geometry}

In this preliminary section, we prove a number of loosely related results in $p$-adic geometry. We continue to fix a prime $p$. Group actions on spaces will mostly be right actions throughout this section.

\medskip

Until further notice, ``adic space" means ``analytic adic space over $\Zp$". In what follows, we freely use the language of diamonds and some standard notation from \cite{scholze-diamonds}. Recall that a diamond is a pro-\'etale sheaf on the site $\mathrm{Perf}$ of characteristic $p$ perfectoid spaces with certain properties. If $X$ is an adic space, the corresponding diamond $X^\lozenge$ comes equipped with a natural map $X^\lozenge \to \Spd\,\Zp$; since $\mathrm{Perf}_{/ \Spd\,\Zp}$ is naturally equivalent to the category $\mathrm{Perfd}$ of all perfectoid spaces, one is free to think of $X^\lozenge$ as a functor on $\mathrm{Perfd}$.  If $X$ is a diamond with a $\underline{G}$-action for some profinite group $G$, we write $X/\underline{G}$ for the quotient sheaf computed as a \emph{pro-\'etale} sheaf.

\begin{lemm}\label{groupspread} Let $X$ be a spatial diamond with a $\underline{G}$-action for some profinite group $G$. Suppose that $G$ acts with finitely many orbits on $\pi_0 X$, and that each connected component of $X$ is a perfectoid space. Then $X$ is a perfectoid space.
\end{lemm}
\begin{proof} Let $X_0$ be some connected component of $X$, and let $x \in X_0$ be any point. Choose some open affinoid perfectoid neighborhood $U \sub X_0$ of $x$. Let $c:|X| \to \pi_0 X$ be the natural map, with $s = c(X_0)$. Writing \[X_0 = \lim_{s \in S \subset \pi_0 X \,\mathrm{clopen}} c^{-1}(S)\] as a cofiltered inverse limit of clopen spatial subdiamonds of $X$ and applying \cite[Proposition 11.23(iii)]{scholze-diamonds}), we deduce that $U$ spreads out to a small open spatial subdiamond $\tilde{U} \sub X$ with $\tilde{U} \cap X_0=U$.  Let $K \sub G$ be the open subgroup stabilizing $\tilde{U}$. Then for any $k \in K$, $\tilde{U} \cap X_0 k = \tilde{U}k \cap X_0 k = (\tilde{U} \cap X_0)k = Uk$ is an affinoid perfectoid space. Since our assumptions on the group action guarantee that the orbit $X_0 K$ is an open spatial subdiamond of $X$, we deduce that $\tilde{U} \cap X_0 K $ is an open spatial subdiamond of $X$ containing $x$, with the property that each connected component of $\tilde{U} \cap X_0 K $ is affinoid perfectoid. By \cite[Lemma 11.27]{scholze-diamonds}, we deduce that $\tilde{U} \cap X_0 K$ itself is affinoid perfectoid. Since $X_0$ and $x$ were arbitrary, we get the result.
\end{proof}

We now turn to some general results on group quotients. Let $X$ be an adic space equipped with an action of a finite group $G$. The coarse quotient $X/G$ always exists in Huber's category $\mathcal{V}$, but in general it may not be an adic space. We need some general results showing that if $X$ is a rigid analytic space or a perfectoid space, then so is $X/G$. The first author already considered this problem in \cite{hansen}, but the results there can be difficult to apply, since they included the assumption that $X$ admits a $G$-invariant affinoid covering, and such coverings can be hard to exhibit in ``real-life'' situations.  Here we obtain much more satisfying and user-friendly results, which don't assume the a priori existence of $G$-invariant affinoid covers.  In the rigid analytic situation we obtain a very general result, cf. Theorem \ref{rigidgpquotient} below. In the perfectoid situation, we need slightly stronger hyptheses, cf. Theorem \ref{perfgpquotient}, but the result is sufficient for our intended applications to Shimura varieties.

\medskip

Let $X$ be a topological space with an action of a \emph{finite} group $G$ by continuous automorphisms. Let $x\in X$ be any point, with stabilizer $H_{x} \sub G$. We say an open neighborhood $U$ of $x$ is \emph{$G$-clean} if $Uh=U$ for all $h\in H_{x}$ and moreover $U\cap Ug=\emptyset$ for all $g\in G\smallsetminus H_{x}$. Note in particular that if $U$ is a $G$-clean neighborhood of $x$, then the natural map 
\[
U\times^{H_{x}}G\overset{(u,g)\mapsto ug}{\longrightarrow}X
\]
is an open embedding, and its image is just the union inside $X$
of $[G:H_{x}]$ many disjoint translates of $U$, so this is an especially
pleasant type of $G$-stable open containing the orbit $xG$.

\begin{lemm}\label{gclean}
Let $X$ be a Hausdorff topological space with a $G$-action.
Then every point $x\in X$ admits a $G$-clean open neighborhood.
\end{lemm}

\begin{proof}
Fix $x\in X$, with stabilizer $H$. Choose coset representatives
$G=\coprod_{1\leq i\leq n}Hg_{i}$ with $g_{1}=1$; the orbit of $x$ is then $\{ x_1, \dots , x_n\}$, with $x_i = xg_i$. Since $X$ is Hausdorff we may choose pairwise disjoint open neighborhoods $U_i^\prime $ of the $x_i$'s. Clearly $g_i^{-1}H g_i$ is the stabilizer of $x_i$, so the open set
\[
U_i = \bigcap_{k \in g_i^{-1}H g_i} U^\prime_i k
\]
contains $x_i$ and is stable under $g_i^{-1}Hg_i$; moreover the $U_i$'s are pairwise disjoint. Now set $V_i = U_i g_i^{-1}$, so $x \in V_i$ and $V_i$ is $H$-stable. Finally, set $W = \bigcap_i V_i$; we claim that $W$ is a $G$-clean open neighborhood of $x$. Indeed, $W$ is $H$-stable since the $V_i$'s are, so it remains to check that if $i  \neq j$, then $Wg_i \cap Wg_j = \emptyset$. But $Wg_i \sub V_i g_i = U_i$ and similarly for $Wg_j$, so $Wg_i \cap Wg_j \sub U_i \cap U_j = \emptyset$, as desired.
\end{proof}

\begin{theo}\label{rigidgpquotient} Let $X$ be a rigid analytic space over some nonarchimedean field $K$ with an action of a finite group $G$. Assume that $X$ is separated, and that for every rank one point $x \in X$, the closure $\overline{ \{x \} } \sub X$ is contained in some open affinoid subspace $U=\Spa(A,A^\circ) \sub X$. Then the categorical quotient $X/G = (|X|/G, (q_\ast \mathcal{O}_X)^G, \cdots)$ is a rigid analytic space, and the natural map $X \to X/G$ is finite. Moreover, the canonical map $X^\lozenge / \underline{G} \to (X/G)^\lozenge$ is an isomorphism.
\end{theo}

 The auxiliary conditions on $X$ in this theorem are satisfied e.g. if $X$ is affinoid, or if $X$ is partially proper.  In particular, the theorem applies whenever $X$ is the analytification of a separated $K$-scheme of finite type.  We would like to emphasize that these auxiliary conditions do not involve the $G$-action in any way. In particular, we are \emph{not} assuming a priori that $X$ admits a covering by $G$-stable affinoid subsets (though, a posteriori, the theorem shows that this is the case).

\begin{proof}
Let $x\in |X|$ be any rank one point, with stabilizer $H_{x}$ and
closure $\overline{\{x\}}\sub |X|$. Let $|X|^{h}$ be the maximal
Hausdorff quotient of $|X|$, and let $\pi : |X| \to |X|^{h}$ be the natural map, so if $x \in |X|$ is any rank one point, then $\overline{\{ x \} } \sub \pi^{-1}(\pi(x))$.\footnote{ One might guess that in fact $\overline{\{ x \} } = \pi^{-1}(\pi(x))$, but this is not clear to us. Indeed, let $|X|^{\nu}$ be the quotient of $|X|$ by the transitive closure of the pre-relation ``$x \sim y$ if $U \cap V \neq \emptyset$ for all open neighborhoods $x \in U, y \in V$". Then $\pi$ naturally factors as a composition of quotient maps $|X| \overset{\tau}{\to} |X|^\nu \overset{q}{\to} |X|^h$. By some standard structure theory of analytic adic spaces, $\tau$ induces a bijection from the rank one points of $|X|$ onto $|X|^\nu$, and $\tau^{-1}(\tau(x)) = \overline{ \{ x \} }$ for any rank one point $x \in |X|$. However, the map $q$ may not be a homeomorphism: for a general topological space $T$, $T^h$ can be obtained by transfinitely iterating the construction $T\rightsquigarrow T^\nu$.  When $|X|$ is \emph{taut}, one can prove that $q$ is a homeomorphism by combining \cite[Lemmas 5.3.4 and 8.1.5]{huber}.} By functoriality of the maximal Hausdorff quotient, $G$ naturally acts on $|X|^h$ and $\pi$ is $G$-equivariant.  By Lemma \ref{gclean} we can choose a $G$-clean open neighborhood $U_{x}\sub |X|^{h}$
of $\pi(x)$. Set $\tilde{U}_{x}=\pi^{-1}(U_x) \sub |X|$,
so $\tilde{U}_{x}$ is a $G$-clean open neighborhood
of $x$ containing $\overline{\{x\}}$.

\medskip

By assumption, we can choose an open affinoid subspace $V_{x}=\mathrm{Spa}(A,A^\circ)\sub X$
containing $\overline{\{x\}}$. Since $X$ is separated, the intersection $\cap_{h\in H_{x}}V_{x}h$ is still affinoid, so after replacing $V_{x}$ by $\cap_{h\in H_{x}}V_{x}h$,
we can assume that $V_{x}$ is $H_{x}$-stable. The intersection $W_{x}=\tilde{U}_{x}\cap V_{x}$
is still a $G$-clean open neighborhood of $x$ containing $\overline{\{x\}}$.
Now, observe that $W_{x}\times^{H_{x}}G\sub X$ is a $G$-stable
open subspace of $X$ containing $\overline{\{x\}}G$ with the crucial
property that
\[
W_{x}/H_{x}\cong\left(W_{x}\times^{H_{x}}G\right)/G\sub X/G
\]
is naturally a rigid analytic space, because $V_{x}/H_{x}\cong\mathrm{Spa}(A^{H_{x}},A^{\circ H_{x}})$
is an affinoid rigid space (e.g. by \cite{hansen}) and $|W_{x}|/H_{x}$ is an open subset of
$|V_{x}|/H_{x}$. Varying over all rank one points $x\in X$, the
spaces $W_{x}/H_{x}$ give an open covering of $X/G$ by rigid analytic
spaces, so $X/G$ is a rigid analytic space, as desired.

\medskip

For finiteness of the map $X\to X/G$, note that $f:W_{x}\to W_{x}/H_{x}$
is finite, since it's the pullback of the finite map $V_{x}\to V_{x}/H_{x}$
along $W_{x}/H_{x}\to V_{x}/H_{x}$. It then suffices to observe that
the pullback of $X\to X/G$ along the open embedding $W_{x}/H_{x}\to X/G$
is given by the map
\[
W_{x}\times^{H_{x}}G\simeq\coprod_{1\leq i\leq n}W_{x}g_{i}\overset{\coprod f\circ g_{i}^{-1}}{\longrightarrow}W_{x}/H_{x},
\]
which is clearly finite.

\medskip

For the last point, it suffices to prove that the canonical maps $V_{x}^\lozenge /\underline{H_{x}} \to (V_{x}/H_{x})^{\lozenge}$ are isomorphisms of pro-\'etale sheaves. We claim that in fact for any Tate $\Zp$-algebra $A$ with an action of a finite group $G$ and a $G$-stable subring of integral elements $A^+$, the canonical map $\Spd(A,A^+) / \underline{G} \to \Spd(A^G,A^{+G})$ is an isomorphism. It suffices to check that $\Spd(A,A^+) \times \underline{G} \rightrightarrows \Spd(A,A^+) $ is a presentation of $\Spd(A^G,A^{+G})$ as a pro-\'etale sheaf. Arguing as in \cite[Proposition 2.1.1]{cgj}, this reduces to the fact that the maps $\Spd(A,A^+) \to \Spd(A^{G},A^{+G})$ and $\Spd(A,A^+) \times \underline{G} \to \Spd(A,A^+) \times_{\Spd(A^G,A^{+G})} \Spd(A,A^+)$ are quasi-pro-\'etale. Since the morphisms in question are separated, this can be checked on rank one geometric points by \cite[Proposition 13.6]{scholze-diamonds}, where it is obvious.
\end{proof}

Unfortunately, the perfectoid variant of the previous theorem is not so clean, primarily because of ``problems'' with the notion of a ``separated'' perfectoid space. For example, for perfectoid spaces over a perfectoid field, the notion introduced in \cite[Definition 5.10]{scholze-diamonds} is too weak for our purposes. The following notion of separation is more than sufficient for our purposes. In what follows, we will frequently use the fact that if $X$ and $Y$ are perfectoid spaces over $\Spa(K,K^+)$ for some affinoid field $(K,K^+)$, then the fiber product $X \times_{\Spa(K,K^+)}Y$ is naturally a perfectoid space. By gluing, this reduces to the claim that this fiber product is naturally affinoid perfectoid if $X$ and $Y$ are each affinoid perfectoid, which is \cite[Corollary 3.6.18]{kedlaya-liu}.

\begin{defi}
\begin{enumerate}
\item A map of perfectoid spaces $Z \to X$ is a \emph{Zariski-closed embedding} if for any open affinoid perfectoid subset $U \sub X$, the map $Z \times_X U \to U$ is a Zariski-closed embedding of affinoid perfectoid spaces in the sense of \cite[\S 2.2]{scho}. We say that an open subset $U$ of a perfectoid space $X$ is \emph{Zariski open} if the inclusion $X \setminus U \to X$ is a Zariski closed embedding.

\item A perfectoid space $X$ over a nonarchimedean field $\Spa(K,K^+)$ is \emph{analytically separated} if the diagonal map $X \to X \times_{\Spa(K,K^+)} X$ is a Zariski-closed embedding.
\end{enumerate}
\end{defi}

We caution the reader this definition of being a Zariski-closed embedding is rather delicate: among other things, it's not clear whether this property can be checked locally on a single affinoid cover of $X$, or whether this property is stable under base change. The key property of analytically separated perfectoid spaces that we will use is part (2) of the following lemma.

\begin{lemm} \label{analyticallyseparatedproperties}
 \begin{enumerate}
\item If a perfectoid space $X$ is analytically separated, then it is separated in the sense of \cite{scholze-diamonds}, i.e. $X^\lozenge \to \Spd(K,K^+)$ is a separated map of v-sheaves.

\item If $X$ is analytically separated, then for any two open affinoid perfectoid subsets $U, V \sub X$, the intersection $U \cap V$ is affinoid perfectoid.
\end{enumerate}
\end{lemm}
\begin{proof} Part (1) is straightforward and left to the reader (and we won't need it anyway). Part (2) is immediate upon writing $U \cap V = (U \times_{\Spa(K,K^+)} V) \times_{X \times_{\Spa(K,K^+)} X,\Delta} X$.
\end{proof}

In practice, analytic separation can often be checked via the following lemma.

\begin{lemm}\label{rigidlimitsep} Let $(X_i)_{i \in I}$ be a cofiltered inverse system of separated rigid analytic spaces over some $\Spa(K,K^\circ)$, and suppose there is some perfectoid space $X_\infty$ such that $X_\infty = \varprojlim_i X_i^{\lozenge}$ as diamonds. Suppose moreover that each $X_i$ is an open subset of the analytification of a projective variety over $K$. Then $X_\infty$ is analytically separated.
\end{lemm}

\begin{proof} By assumption, we can choose open immersions $X_i \to V_{i}^{\an}$ for some projective varieties $V_i$.  Let $U \sub X_\infty \times_{\Spa(K,K^\circ)} X_\infty$ be some open affinoid perfectoid subset. Set \[ W_i = U \times_{X_i \times_{\Spa(K,K^\circ)} X_i,\Delta} X_i \cong U \times_{V_{i}^{\an} \times_{\Spa(K,K^\circ)} V_{i}^{\an},\Delta} V_{i}^{\an}. \] A priori, we are computing this fiber product as diamonds. However, by the subsequent lemma, $W_i$ is affinoid perfectoid and the resulting map $W_i \to U$ is a Zariski-closed embedding. Then $U \times_{X_\infty \times_{\Spa(K,K^\circ)} X_\infty,\Delta} X_\infty = \varprojlim_i W_i$ is affinoid perfectoid, and $\varprojlim_i W_i \to U$ is a cofiltered limit of Zariski-closed embeddings. Since any cofiltered limit of Zariski-closed embeddings with fixed target is a Zariski-closed embedding, we get the result.
\end{proof}

\begin{lemm} Let $Y \to X$ be a closed immersion of quasi-projective varieties over a nonarchimedean field $K$, and let $Z$ be any perfectoid space equipped with a map $f:Z \to X^{\an}$. Then the diamond $W= Z \times_{X^{\an}} Y^{\an}$ is a perfectoid space, and the natural map $W\to Z$ is a Zariski-closed embedding.
\end{lemm}

\begin{proof} Unwinding the definitions, it suffices to prove that if $Z$ is affinoid perfectoid, then $W= Z \times_{X^{\an}} Y^{\an} \to Z$ is a Zariski-closed embedding of affinoid perfectoid spaces.

\medskip

Replacing $X$ by its closure in some projective space, and replacing $Y$ by its closure in $X$, we can assume that $Y \to X$ is a closed immersion of projective varieties. Let $\mathcal{I} \sub \mathcal{O}_{X^{\an}}$ be the ideal sheaf cutting out $Y^\an$. By rigid GAGA and the projectivity of $X$, we can choose a vector bundle $\mathcal{E}$ on $X^\an$ together with a surjection $\mathcal{E} \twoheadrightarrow \mathcal{I}$. Then $f^\ast \mathcal{E}$ is naturally a vector bundle on $Z$, and the image of the natural map $f^{\ast}\mathcal{E} \to \mathcal{O}_Z$ is just the ideal sheaf generated by $f^{-1} \mathcal{I}$. However, $Z$ is affinoid perfectoid, so $f^\ast\mathcal{E}$ is generated by its global sections, which are just a finitely generated projective $\mathcal{O}_Z(Z)$-module. In paticular, if $e_1, \dots, e_n \in H^0(Z,f^\ast \mathcal{E})$ is any set of generators, then their images in $\mathcal{O}_Z(Z)$ generate an ideal $I$ corresponding to the ideal sheaf generated by $f^{-1} \mathcal{I}$. Let $W\sub Z$ be the Zariski-closed subset cut out by $I$. It is then easy to see that $W$ represents the fiber product claimed in the statement of the lemma.
\end{proof}

\begin{theo}\label{perfgpquotient} Let $X$ be a perfectoid space over a nonarchimedean field, with an action of a finite group $G$. Assume that $X$ is analytically separated, and that for every rank one point $x \in X$, the closure $\overline{ \{x \} } \sub X$ is contained in some open affinoid perfectoid subspace $U=\Spa(A,A^+) \sub X$. 

Then the categorical quotient $X/G$ is a perfectoid space, and the natural map $q: X \to X/G$ is affinoid in the (weak) sense that any point $y \in X/G$ admits a neighborhood basis of open affinoid perfectoid subsets $Y\sub X$ whose preimages $q^{-1}(Y)$ are affinoid perfectoid. Moreover, the canonical morphism $X^\lozenge / \underline{G} \to (X/G)^{\lozenge}$ is an isomorphism.
\end{theo}
\begin{proof} The first portion of the proof is nearly identical to the proof of Theorem \ref{rigidgpquotient}, but we repeat the details for the reader's convenience.

\medskip

Let $x\in X$ be any rank one point, with stabilizer $H_{x}$ and
closure $\overline{\{x\}}\subset X$. Let $|X|^{h}$ be the maximal
Hausdorff quotient of $|X|$, and let $\pi : |X| \to |X|^{h}$ be the natural map, so if $x \in |X|$ is any rank one point, then $\overline{\{ x \} } \subseteq \pi^{-1}(\pi(x))$. By functoriality of the maximal Hausdorff quotient, $G$ naturally acts on $|X|^h$ and $\pi$ is $G$-equivariant.  By Lemma \ref{gclean} we can choose a $G$-clean open neighborhood $U_{x}\sub |X|^{h}$
of $\pi(x)$. Let $\tilde{U}_{x}$ be the preimage of $U_{x}$ in $|X|$,
so $\tilde{U}_{x}$ is a $G$-clean open neighborhood
of $x$ containing $\overline{\{x\}}$.

\medskip

By assumption, we can choose an open affinoid perfectoid subspace $V_{x}=\mathrm{Spa}(A,A^+)\sub X$ containing $\overline{\{x\}}$. Since $X$ is analytically separated, the intersection $\cap_{h\in H_{x}}V_{x}h$ is affinoid perfectoid by Lemma \ref{analyticallyseparatedproperties}.(2), so after replacing $V_{x}$ by $\cap_{h\in H_{x}}V_{x}h$,
we can assume that $V_{x}$ is $H_{x}$-stable. The intersection $W_{x}=\tilde{U}_{x}\cap V_{x}$
is still a $G$-clean open neighborhood of $x$ containing $\overline{\{x\}}$.
Now, observe that $W_{x}\times^{H_{x}}G\subset X$ is a $G$-stable
open subspace of $X$ containing $\overline{\{x\}}G$ with the crucial
property that
\[
W_{x}/H_{x}\cong\left(W_{x}\times^{H_{x}}G\right)/G \sub X/G
\]
is naturally a perfectoid space, because $V_{x}/H_{x}\cong\mathrm{Spa}(A^{H_{x}},A^{+ H_{x}})$
is an affinoid perfectoid space by \cite[Theorem 1.4]{hansen} and $|W_{x}|/H_{x}$ is an open subset of
$|V_{x}|/H_{x}$. Varying over all rank one points $x\in X$, the
spaces $W_{x}/H_{x}$ give an open covering of $X/G$ by perfectoid
spaces, so $X/G$ is a perfectoid space, as desired. 

\medskip

To see that $q$ is affinoid, let $y \in X/G$ be any point, so $y$ is contained in some $W_x/H_x$. Let $Y \sub W_x/H_x \sub X/G$ be any open subset containing $y$ such that $Y$ is a rational subset of $V_x/H_x$. The set of such $Y$'s is clearly a neighborhood basis of $y$. Moreover, $q^{-1}(Y)$ is a finite disjoint union of copies of the preimage of $Y$ in $V_x$, but the latter preimage is a rational subset of $V_x$, and hence is affinoid perfectoid, so $q^{-1}(Y)$ is affinoid perfectoid. Varying $y$, we get the claim. 

\medskip

The last point follows exactly as in the proof of Theorem \ref{rigidgpquotient}.
\end{proof}
  
In the next section, we will often be in a situation where we have a morphism between two inverse systems of Shimura varieties for some closely related Shimura data. In the remainder of this section, we prove some results which will allow us to transfer information from one inverse system to the other.

\begin{lemm} \label{twotowers} Let $(X_i)_{i \in I} \overset{f_i}{\longrightarrow} (Y_i)_{i \in I}$ be a morphism of cofiltered inverse systems of locally Noetherian adic spaces. Assume moreover that the maps $f_i$ and the transition maps in the inverse systems are all finite maps, and that $Y_\infty = \varprojlim_i Y_i^{\lozenge}$ is perfectoid. 

Then $X_{\infty} = \varprojlim_i X_i^{\lozenge}$ is perfectoid, and the morphism $f_\infty: X_\infty \to Y_\infty$ is quasicompact. Moreover, if $U \sub Y_\infty$ is an open affinoid perfectoid subset which arises as the preimage of an open affinoid $U_i \sub Y_i$ for some $i$, then $f_{\infty}^{-1}(U) \sub X_\infty$ is also affinoid perfectoid. Finally, $f_{\infty}$ is affinoid in the sense of Theorem \ref{perfgpquotient}.
\end{lemm}

With more effort, one can show that the morphism $f_\infty$ is proper and quasi-pro-\'etale in the sense of \cite{scholze-diamonds}. We will not need this.

\begin{proof} Without loss of generality, we may assume that $I$ contains an initial element $0$. Next, observe that 
\begin{align*}
X_\infty & \cong \varprojlim_j X_\infty \times_{Y_j} Y_{\infty} \\
 & \cong \varprojlim_{i \geq j} X_i \times_{Y_j} Y_{\infty} \\
  & \cong \varprojlim_{i} X_i \times_{Y_i} Y_{\infty}
\end{align*} 
using the cofinality of the diagonal to get the last line. Choose an open affinoid subset $U_0 \sub Y_0$ with preimages $U_i \sub Y_i$, $W_i \sub X_i$, $U_\infty \sub Y_\infty$, $W_\infty \sub X_\infty$. To prove the first part of the theorem, it suffices to prove that $W_\infty$ is a perfectoid space. This can be checked locally on some covering of $U_\infty$ by open affinoid perfectoid subsets $V = \Spa(R,R^+) \sub U_\infty$. By our assumptions, the natural maps $W_i \to U_i$ are finite maps of affinoid adic spaces, so in particular $\mathcal{O}^+(U_i) \to \mathcal{O}^+(W_i)$ is an integral ring map. By general nonsense, the fiber product $X_i \times_{Y_i} V = W_i \times_{U_i} V$ is computed as $\mathrm{Spd}(S,S^+)$, where $S = R \otimes_{\mathcal{O}(U_i)} \mathcal{O}(W_i)$ (topologized in the usual way) and $S^+$ is the integral closure of $\mathrm{im}(R^+ \otimes_{\mathcal{O}^+(U_i)} \mathcal{O}^+(W_i) \to S)$ in $S$. In particular, $R^+ \to S^+$ is an integral ring map, so the subsequent lemma implies that $W_i \times_{U_i} V$ is an affinoid perfectoid space. Passing to the limit over $i$, we deduce that $W_\infty \times_{U_\infty} V$ is an affinoid perfectoid space, and then varying over all choices of $U_0 \sub Y_0$ and $V \sub U_\infty$ as above, we conclude that $X_\infty$ is a perfectoid space. 

\medskip
   
Quasicompactness of $f_{\infty}$ is clear. For the remaining claims of the theorem, choose some $U_i \sub Y_i$ and $U \sub Y_\infty$ as in the statement of the claim, and let $U_j \sub Y_j$ and $W_j \sub X_j$ denote the evident preimages for all $j \geq i$. Arguing as in the first part of the proof, we see that $f_{\infty}^{-1}(U)=\varprojlim_{j \geq i} W_j \times_{U_j} U$ and that $W_j \times_{U_j} U$ is an affinoid perfectoid space for any $j \geq i$. Passing to the limit over $j$ gives the claim. Affinoidness of $f_{\infty}$ now follows from Lemma \ref{affinoidcover} below.
\end{proof}

In the course of this proof, we crucially used the following result, which is essentially just a rephrasing of a theorem of Bhatt-Scholze.

\begin{lemm} \label{integralperfectoid} Let $(R,R^+) \to (S,S^+)$ be a map of Tate-Huber pairs such that $R$ is a perfectoid Tate ring and the ring map $R^+ \to S^+$ is  integral. Then the diamond $\Spd(S,S^+)$ is an affinoid perfectoid space.
\end{lemm}
\begin{proof}Choose a pseudouniformizer $\varpi \in R^+$. Since $R^+$ is integral perfectoid and $R^+ \to S^+$ is an integral ring map, \cite[Theorem 1.16(1)]{bhatt-scholze} guarantees the existence of an integral perfectoid $S^+$-algebra $S^+_{\mathrm{perfd}}$ such that any map from $S^+$ to an integral perfectoid ring factors uniquely through the map $S^+ \to S^+_{\mathrm{perfd}}$. Set $T=S^+_{\mathrm{perfd}}[1/\varpi]$, and let $T^+ \subset T$ be the integral closure of $S^+_{\mathrm{perfd}}$ in $T$. Then $T$ is a perfectoid Tate ring, and the natural map $(S,S^+) \to (T,T^+)$ induces a bijection $\mathrm{Hom}( (T,T^+), (A,A^+) ) \cong \mathrm{Hom} ((S,S^+),(A,A^+))$ for any perfectoid Tate-Huber pair $(A,A^+)$. This shows that $\Spd(S,S^+)\cong \Spd(T,T^+)$ is affinoid perfectoid, as desired.
\end{proof}

We also used the following result.

\begin{lemm} \label{affinoidcover} Let $(X_i)_{i \in I}$ be a cofiltered inverse system of locally Noetherian adic spaces with finite transition maps. Assume that $X =\varprojlim X_i^{\lozenge}$ is a perfectoid space. Then $X$ has a neighborhood basis of open affinoid perfectoid subsets $W \subset X$ which are preimages of open affinoids $W_i \subset X_i$ at (variable) finite levels.
\end{lemm}
\begin{proof} Without loss of generality, we can assume that $I$ has an initial object $0$. The problem is local on $X_0$, so replacing $X_0$ by an open affinoid subset and using the finiteness of the maps in the tower, we can also assume that all $X_i$'s are affinoid, say with $X_i = \Spa(B_i, B_i^+)$. Let $(B,B^+)$ be the completed direct limit of the system $(B_i,B_i^+)$, so $X \cong \Spd(B,B^+)$. Now, let $\mathcal{W}$ be the set of rational subsets $W \subset X$ which are contained in some open affinoid perfectoid subset of $X$. Then any $W \in \mathcal{W}$ is affinoid perfectoid, and  elements of $\mathcal{W}$ clearly form a neighborhood basis of $X$. On the other hand, any rational subset of $X$, and in particular any element of $\mathcal{W}$, is the preimage of a rational subset of $X_i$ for some large $i$ by standard approximation arguments.
\end{proof}

In applications, we will usually care about inverse systems with the following restrictive properties.

\begin{defi}\label{goodtower} Fix a nonarchimedean field $K$. A \emph{good tower} is a cofiltered inverse system of locally Noetherian adic spaces $(X_i)_{i \in I}$ over $\Spa K$ with the following properties.
\begin{enumerate}

\item Each $X_i$ is the analytification of a projective variety over $K$, and the transition maps are finite.

\item The inverse limit $X = \varprojlim_{i} X_i^{\lozenge}$ is a perfectoid space.

\item There exists a pair of coverings of $X$ by open affinoid perfectoid subsets $U_j, V_j$ such that $\overline{U_j} \sub V_j$ for all $j$, and such that for each $j$, $U_j$ and $V_j$ occur as the preimages of some open affinoids $U_{j,i_j}, V_{j,i_j} \sub X_{i_j}$ for some $i_j\in I$.
\end{enumerate}
\end{defi}

\noindent The point of this definition is captured in the following proposition.

\begin{prop} \label{goodtowerproperties}
\begin{enumerate} 
\item Let $(Y_i)_{i \in I}$ be a good tower. If $(X_i)_{i \in I} \overset{f_i}{\longrightarrow} (Y_i)_{i \in I}$ is any map of cofiltered inverse systems such that the morphisms $f_i$ are finite, then $(X_i)_{i \in I}$ is a good tower.

\item If $(X_i)_{i \in I}$ is a good tower with an action of a finite group $G$, then the categorical quotient $X/G$ is a perfectoid space and $X/G \cong \varprojlim_{i} X_i/G$.
\end{enumerate}
\end{prop}

Note that in part (2), we are not claiming that $(X_i / G)_{i \in I}$ is a good tower: it's not clear to us whether condition (3) is preserved. 

\begin{proof}
For part (1), let $f:X \to Y$ denote the map between the limits of the towers. Note that since $X_i \to Y_i$ is finite, the tower $(X_i)_{i \in I}$ satisfies condition (1) of Definition \ref{goodtower} by rigid GAGA. Conditions (2) and (3) then follow from Lemma \ref{twotowers}. Indeed, (2) is immediate, and (3) follows from the observation that if $U_j \sub V_j \sub Y$ are open affinoid perfectoid subsets pulled back from some finite-level affinoids $U_{j,i_j}, V_{j,i_j} \sub Y_i$, then $f^{-1}(U_j)$ is affinoid perfectoid by Lemma \ref{twotowers} and is clearly the preimage of the affinoid $f_{i_j}^{-1}(U_{j,i_j}) \sub X_{i_j}$ (and similarly for the $V_j$'s). Finally, the condition on closures follows from the inclusions $\overline{f^{-1}(U_j)} \sub f^{-1}(\overline{U_j}) \sub f^{-1}(V_j)$.

\medskip

For part (2), $X/G$ is perfectoid by Theorem \ref{perfgpquotient}, since by design the limit of a good tower satisfies the conditions of that Theorem. Indeed, the limit of any good tower is analytically separated by Lemma \ref{rigidlimitsep}. Moreover, if $U_j, V_j \sub X$ are as in the definition of a good tower, then any rank one point $x \in X$ is contained in some $U_j$, in which case $\overline{ \{ x \} } \sub \overline{U_j} \subset V_j$.

\medskip

It remains to check that the natural map $f: X/G \to \varprojlim_{i} X_i/G$ is an isomorphism of diamonds. The source and target of this map are spatial diamonds, so the map is automatically qcqs. Thus, by \cite[Lemma 11.11]{scholze-diamonds}, it suffices to prove that $f$ induces a bijection on $(C,C^+)$-points for every algebraically closed perfectoid field $C$ with an open and bounded valuation subring $C^+ \sub C$.  In what follows, we will freely use the fact that $(C,C^+)$-points can be computed ``naively": if $X$ is a pro-\'etale sheaf with a $\underline{G}$-action for some profinite group $G$ and $X/\underline{G}$ denotes the quotient as pro-\'etale sheaves, then $X(C,C^+)/G \cong (X/\underline{G})(C,C^+)$. This is an easy consequence of the fact that any pro-\'etale cover of a geometric point $(C,C^+)$ has a section.\footnote{More generally, if $\mathcal{F}$ is a presheaf of sets on a site $\mathcal{C}$, and $X \in \mathcal{C}$ is any object with the property that every covering of $X$ admits a section, then the natural map $\mathcal{F}(X) \to \mathcal{F}^{sh}(X)$ is a bijection, where $(-)^{sh}$ denotes sheafification. This is easy and left to the reader.} 

\medskip

For surjectivity, let $(x_i \in X_i (C,C^+)/G)_{i \in I}$ be any inverse system of points. Let $W_i \sub X(C,C^+)$ be the preimage of $x_i$. Since $W_i\cong \varprojlim_{j} W_{i,j}$ where $W_{i,j} \sub X_j(C,C^+)$ is the preimage of $x_i$, and each $W_{i,j}$ is finite and nonempty (use that $X_j \to X_i$ is finite), $W_i$ naturally has the structure of a (non-empty) profinite set. Then $W=\varprojlim_{i} W_i$ is an inverse limit of non-empty compact Hausdorff spaces, and thus is non-empty. Any choice of $x \in W \sub X(C,C^+)$ maps to the inverse system $(x_i)_{i \in I}$.

\medskip

For injectivity, let $x,y \in X(C,C^+)$ be two elements with the same image in $\varprojlim_i X_i(C,C^+)/G$. Let $x_i, y_i \in X_i (C,C^+)$ be the images of $x$ and $y$, and let $G_i \subset G$ be the set $g \in G$ with $gx_i = y_i$. Then $G_i$ is nonempty by assumption, and $G_j \to G_i$ is injective for all $j \geq i$, so $\varprojlim_i G_i$ is nonempty. Choosing any $g \in \varprojlim_i G_i$, we then have $gx=y$, as desired.
\end{proof}

\subsection{Perfectoid Shimura varieties of Hodge type}

We now return to Shimura varieties. Let $(G,X)$ be a Shimura datum of Hodge type, with reflex field $E$ and Hodge cocharacter $\mu$. For any open compact subgroup $K \sub G(\A_f)$, we write $Sh_K(G,X)$ for the canonical model of the associated Shimura variety; this is a normal quasi-projective scheme over $E$. This has a canonical projective minimal compactification $Sh^{\ast}_K(G,X)$, which is also normal. Fix a prime $\mathfrak{p}$ of $E$ lying over $p$, and let $\mathcal{X}_K$, resp. $\mathcal{X}^{\ast}_K$ denote the rigid analytic space over $E_{\mathfrak{p}}$ associated with $Sh_K(G,X) \otimes_E E_\mathfrak{p}$, resp. $Sh^{\ast}_K(G,X) \otimes_E E_\mathfrak{p}$. As $K$ varies, these spaces form a pair of inverse systems with finite transition maps, and compatible open immersions $\mathcal{X}_K \to \mathcal{X}_{K}^{\ast}$. Recall the (rigid analytic) flag variety $\mathscr{F}\! \ell_{G,\mu}$ attached to $(G,X)$, as defined over $E_\mathfrak{p}$ in \cite[\S 2.1]{caraiani-scholze}.

\begin{prop} \label{hodgetypeperfectoid} Fix any open compact subgroup $K^p \sub G(\A_{f}^{p})$. Then $\mathcal{X}^{\ast}_{K^p} = \varprojlim_{K_p} \mathcal{X}^{\ast,\lozenge}_{K^p K_p}$ is a perfectoid space, and there is a $G(\Qp)$-equivariant Hodge-Tate period map $\pi_{HT}: \mathcal{X}^{\ast}_{K^p} \to \mathscr{F}\! \ell_{G,\mu}$ which is functorial in the tame level. 

  Moreover, $\mathcal{X}^{\ast}_{K^p}$ is analytically separated, and we can find a pair of coverings by finitely many open affinoid perfectoid subsets $U_i, V_i \sub \mathcal{X}^{\ast}_{K^p}$ such that $\overline{U_i} \sub V_i$ for all $i$ and such that $U_i$ and $V_i$ arise as the preimages of some open affinoid subsets of some $\mathcal{X}^{\ast}_{K^p K_p}$. 
  
  In particular, for any cofinal system of open compact subgroups $K_p \sub G(\Qp)$, $(\mathcal{X}^{\ast}_{K^p K_p})_{K_p}$ is a good tower (over $E_\mathfrak{p}$) in the sense of Definition \ref{goodtower}.
\end{prop}

Note that $\mathcal{X}^{\ast}_{K^p}$ may not coincide with the ``ad hoc" compactification $\mathcal{X}^{\underline{\ast}}_{K^p}$ constructed in [Sch15], although by construction there is certainly a map $\mathcal{X}^{\ast}_{K^p} \to \mathcal{X}^{\underline{\ast}}_{K^p}$. The fact that we can prove this result for the genuine minimal compactification relies crucially on the perfectoidization theorem of Bhatt-Scholze recalled in Lemma \ref{integralperfectoid}.

\begin{proof} Fix a closed embedding $\iota: (G,X) \to (\mathrm{GSp}_{2g}, \mathfrak{H}^{\pm}_{g})$ into a Siegel Shimura datum. For any open compact subgroup $K\sub \mathrm{GSp}_{2g}(\Qp)$, let $\mathcal{S}_K$, resp. $\mathcal{S}^{\ast}_K$ denote the rigid analytic space over $E_{\mathfrak{p}}$ associated with $Sh_K(\mathrm{GSp}_{2g},\mathfrak{H}^{\pm}_{g}) \otimes_{\Q} E_\mathfrak{p}$, resp. $Sh^{\ast}_K(\mathrm{GSp}_{2g},\mathfrak{H}^{\pm}_{g}) \otimes_{\Q} E_\mathfrak{p}$.  By \cite[Theorem 3.3.18]{scho}, $\varprojlim_{K_p} \mathcal{S}^{\ast}_{K^p K_p}$ is a perfectoid space for any open compact subgroup $K^p \sub \mathrm{GSp}_{2g}(\Qp)$ contained in some conjugate of a principal congruence subgroup of level $\geq 3$. However, this last condition can easily be removed using \cite[Theorem 1.4]{hansen}, noting in particular that $\mathcal{S}^{\ast}_{K^p}$ is covered by finitely many $\mathrm{GSp}_{2g}(\Qp)$-translates of a certain open affinoid perfectoid subset $\mathcal{S}^{\ast}_{K^p}(\epsilon)_a$,\footnote{This subset is denoted $\mathcal{X}^{\ast}_{\Gamma(p^\infty)}(\epsilon)_a$ in \cite[\S 3]{scho}.} and that these subsets are invariant under the action of $K'^p / K^p$ for any normal inclusion $K^p \sub K'^p$ of tame level groups. 

\medskip

The chosen embedding $\iota$ gives rise to compatible finite maps $\mathcal{X}_{K \cap G(\A_f)} \to \mathcal{S}_{K}$ for any $K\sub \mathrm{GSp}_{2g}(\A_f)$ as above, which naturally extend to compatible finite morphisms $\mathcal{X}^{\ast}_{K \cap G(\A_f)} \to \mathcal{S}^{\ast}_{K}$. Now, choose any $K^p \sub G(\A_{f}^{p})$ as in the proposition, and choose an open compact $K'^p \sub \mathrm{GSp}_{2g}(\A_{f}^{p})$ such that $K^p \sub K'^p$. Choosing a cofinal set of (neat) open compact subgroups $K_0 \supseteq K_1 \supseteq K_2 \cdots $ in $\mathrm{GSp}_{2g}(\Qp)$, we get a map of inverse systems $(\mathcal{X}^{\ast}_{K^p \iota^{-1}(K_n)})_{n\geq 0}  \to (\mathcal{S}^{\ast}_{K'^p K_n} )_{n \geq 0}$ satisfying all the hypotheses of Lemma \ref{twotowers}. Applying that lemma, we deduce that $\mathcal{X}^{\ast}_{K^p}$ is a perfectoid space and the natural map $f: \mathcal{X}^{\ast}_{K^p} \to \mathcal{S}^{\ast}_{K'^p}$ is quasicompact. Moreover, $\mathcal{X}^{\ast}_{K^p}$ is analytically separated by Lemma \ref{rigidlimitsep}.

\medskip

Now choose some $0<\epsilon < \epsilon'< 1/2$ and finitely many $g_i \in \mathrm{GSp}_{2g}(\Qp)$ such that the translates $\mathcal{S}^{\ast}_{K'^p}(\epsilon)_a \cdot g_i$ cover $\mathcal{S}^{\ast}_{K'^p}$. Note that any such translate is the preimage of an open affinoid subset of some $\mathcal{S}^{\ast}_{K'^p K_n}$, so again by Lemma \ref{twotowers} we see that the preimages 
\[ U_i = f^{-1}(\mathcal{S}^{\ast}_{K'^p}(\epsilon)_a \cdot g_i) \sub V_i = f^{-1}(\mathcal{S}^{\ast}_{K'^p}(\epsilon')_a \cdot g_i)\]
are affinoid perfectoid and give open covers of $\mathcal{X}^{\ast}_{K^p}$, and arise by pullback from some finite level. Moreover, $\overline{\mathcal{S}^{\ast}_{K'^p}(\epsilon)_a\cdot g_i} \subset \mathcal{S}^{\ast}_{K'^p}(\epsilon')_a\cdot g_i$ for any $\epsilon < \epsilon' < 1/2$, and clearly $\overline{U_i} \sub f^{-1}(\overline{\mathcal{S}^{\ast}_{K'^p}(\epsilon)_a\cdot g_i})$, so we conclude that $\overline{U_i} \sub V_i$ as desired.

\medskip

The Hodge-Tate period map is the composition of the natural map $\mathcal{X}^{\ast}_{K^p} \to \mathcal{X}^{\underline{\ast}}_{K^p}$ with the (previously known) Hodge-Tate period map $\mathcal{X}^{\underline{\ast}}_{K^p} \to \mathscr{F}\! \ell_{G,\mu}$, cf. \cite[Theorem 3.3.1]{arizona} for a discussion of the latter (the argument there also works to construct $\pi_{HT}: \mathcal{X}^{\ast}_{K^p} \to \mathscr{F}\! \ell_{G,\mu}$ without the use of ad hoc compactifications).
\end{proof}

For later use, we also record an extremely mild generalization of this result.

\begin{coro} \label{hodgetypeperfectoidgen} For any open compact subgroup $K \sub G(\A_f)$ and any cofinal system of open compact subgroups $K_p \sub G(\Q_p)$, $(\mathcal{X}^{\ast}_{K \cap K_p})_{K_p}$ is a good tower (over $E_\mathfrak{p}$) in the sense of Definition \ref{goodtower}.
\end{coro}

Here and in what follows, we adopt the following notational convention: if $G$ is an algebraic group over $\Q$, $H$ is a subgroup of $G(\A_f)$, and $K_p$ is a subgroup of $G(\Qp)$, then $H \cap K_p$ denotes the group of elements $h \in H$ whose image in $G(\Qp)$ lies in $K_p$. In other words, $H \cap K_p$ is short for $H \cap (G(\A_{f}^{p})K_p)$. We hope this doesn't cause any confusion.

\begin{proof} Let $K^p \sub G(\A_{f}^{p})$ denote the image of $K$ along the natural projection. Then $K \cap K_p$ has finite index in $K^p K_p$, so we get natural finite morphisms $\mathcal{X}^{\ast}_{K \cap K_p} \to \mathcal{X}^{\ast}_{K^p K_p}$ which compile into a map of towers $(\mathcal{X}^{\ast}_{K \cap K_p})_{K_p} \to (\mathcal{X}^{\ast}_{K^p K_p})_{K_p}$. Since the target is a good tower by the previous proposition, we may apply Proposition \ref{goodtowerproperties}(i) to conclude.
\end{proof}

\subsection{Perfectoid Shimura varieties of pre-abelian type}

In this section we change notation slightly. Given a Shimura datum $(G,X)$ and an open compact subgroup $K \sub G(\A_f)$, we write $Sh_K(G,X)$ for the associated Shimura variety regarded as a quasi-projective variety \emph{over} $\C$, and $Sh_K^{\ast}(G,X)$ for its projective minimal compactification. For a (usually implicit) choice of connected component $X^+ \sub X$, we write $Sh_K(G,X)^0$ for the connected component of $Sh_K(G,X)$ whose analytification is the image of the natural map 
\[ X^+ \times \{ e \} \to G(\Q)_+ \backslash (X^+ \times G(\A_f)) / K \cong Sh_K(G,X)^{\mathrm{an}}, \]
and we write $Sh_K^{\ast}(G,X)^0$ for the Zariski closure of $Sh_K(G,X)^0$ in $Sh_K^{\ast}(G,X)$. Note that since $Sh_K^{\ast}(G,X)$ is normal, the map $\pi_0 Sh_K(G,X) \to \pi_0 Sh_K^{\ast}(G,X)$ is a homeomorphism.

\medskip

Now, fix once and for all an isomorphism $\C \simeq \C_p$ (for simplicity), and let $C/\C_p$ be a complete algebraically closed extension of nonarchimedean fields. All of the following results hold for any choice of $C$. We write $\mathcal{X}_K^{\ast}(G,X)$ for the rigid analytic space associated with $Sh_K^{\ast}(G,X) \otimes_{\C} C$.  Similarly, we get rigid analytic spaces $\mathcal{X}_K(G,X)$, $\mathcal{X}_K(G,X)^0$, $\mathcal{X}_K^{\ast}(G,X)^0$ with the obvious meanings.

\medskip

For any fixed open compact subgroup $K^p \sub G(\A_{f}^p)$, define
\[\mathcal{X}_{K^p}^{\ast}(G,X)=\varprojlim_{K_p \sub G(\Q_p)\,\mathrm{open\,compact}} \mathcal{X}_{K^pK_p}^{\ast}(G,X)^\lozenge\]
where the inverse limit is taken in the category of diamonds over $\Spd C$. We also write $\mathcal{X}_{K^p}(G,X)$, $\mathcal{X}_{K^p}^{\ast}(G,X)^0$, and $\mathcal{X}_{K^p}(G,X)^0$ for the obvious variants.

\begin{prop}\label{going sideways} Maintain the above notation. The following conditions on a Shimura datum $(G,X)$ are equivalent.
\begin{enumerate}

\item The diamond $\mathcal{X}_{K^p}^{\ast}(G,X)$ is a perfectoid space for any choice of $K^p$.

\item The diamond $\mathcal{X}_{K^p}^{\ast}(G,X)^0$ is a perfectoid space for any choice of $K^p$.

\end{enumerate}
We say the Shimura datum $(G,X)$ satisfies \emph{Property $\mathcal{P}$} if either of these equivalent conditions holds.

\end{prop}

\begin{proof} 
(1) implies (2): In general, $\mathcal{X}_{K^p}^{\ast}(G,X)^0$ is an inverse limit of open-closed subfunctors $\mathcal{X}_i \sub \mathcal{X}_{K^p}^{\ast}(G,X)$. Therefore, if $\mathcal{X}_{K^p}^{\ast}(G,X)$ is perfectoid and $U \sub \mathcal{X}_{K^p}^{\ast}(G,X)$ is any open affinoid perfectoid subset, then $U \cap \mathcal{X}_{K^p}^{\ast}(G,X)^0 = \varprojlim_i U \cap \mathcal{X}_i$ and each $U \cap \mathcal{X}_i$ is affinoid perfectoid, so $U \cap \mathcal{X}_{K^p}^{\ast}(G,X)^0$ is affinoid perfectoid. Varying $U$ then gives the result.

\medskip

(2) implies (1): Choose any open compact subgroup $K_p \subset G(\Qp)$, so the diamond $\mathcal{X}_{K^p}^{\ast}(G,X)$ has a natural $\underline{K_p}$-action. Then $K_p$ acts with finitely many open orbits on the profinite set $ \pi_0 \mathcal{X}_{K^p}^{\ast}(G,X) \cong G(\Q)_+ \backslash G(\A_{f}) / K^p$ (by \cite[Theorem 5.1]{bo}). Moreover, each connected component of $\mathcal{X}_{K^p}^{\ast}(G,X)$ is isomorphic to $\mathcal{X}_{g K^p g^{-1}}^{\ast}(G,X)^0$ for some $g \in G(\A^p_f)$, and in particular is perfectoid. By Lemma \ref{groupspread}, we deduce that $\mathcal{X}_{K^p}^{\ast}(G,X)$ is a perfectoid space, as desired.
\end{proof}

We also need to work with connected Shimura varieties. Let $(G,X^+)$ be a connected Shimura datum. If $\Gamma \subset G(\Q)_+$ is an arithmetic subgroup, then the quotient $\Gamma \backslash X^+$ is the analytification of a connected normal quasiprojective complex variety, defined uniquely up to unique isomorphism, which we denote by $Sh_{\Gamma}(G,X^+)$. Again, this has a canonical minimal compactification $Sh_{\Gamma}^{\ast}(G,X^+)$, which is a connected normal projective variety. If $\Gamma$ is torsion-free, then $Sh_{\Gamma}(G,X^+)$ is smooth. Again, we denote the associated rigid analytic spaces over $C$ by $\mathcal{X}_{\Gamma}^{\ast}(G,X^+)$, etc.

\begin{defi} 
We say a connected Shimura datum $(G,X^+)$ satisfies Property $\mathcal{P}$ if for every arithmetic subgroup $\Gamma \sub G^{ad}(\Q)^+$, the diamond 
\[ 
\mathcal{X}_{\Gamma,\infty}^{\ast}(G,X^+) := \varprojlim_{K_p \subset G(\Qp)\,\mathrm{open\,compact}}  \mathcal{X}_{\Gamma \cap K_p}^{\ast}(G,X^+)^{\lozenge}
\] 
is a perfectoid space.
\end{defi}

In this statement, recall our notational convention that $\Gamma \cap K_p$ is shorthand for $\Gamma \cap (G(\mathbb{A}_{f}^{p})K_p)$ (cf. the discussion following Corollary \ref{hodgetypeperfectoidgen}).

\begin{prop}\label{going up}
Let $(G,X)$ be a Shimura datum or a connected Shimura datum. Suppose that $(G^{ad},X^+)$ satisfies Property $\mathcal{P}$. Then $(G,X)$ satisfies Property $\mathcal{P}$.
\end{prop}

\begin{proof} Let $\pi: G \to G^{ad}$ denote the natural map. When $(G,X)$ is a connected Shimura datum and $\G \sub G(\Q)^+$ is an arithmetic subgroup, then $\mc{X}_{\G,\infty}(G,X^+) = \mc{X}_{\pi(\G),\infty}(G^{ad},X^+)$ and the result follows, so let $(G,X)$ be a Shimura variety. By Proposition \ref{going sideways} it is enough to show that $\mathcal{X}_{K^p}^{\ast}(G,X)^0$ is perfectoid for any $K^p \sub G(\A_f^{p})$.  Let $\Gamma =  G^{ad}(\Q)^+ \cap K$ be a choice of congruence subgroup for some open compact subgroup $K \sub G^{ad}(\A_f)$ with the property that $\pi(K^p) \sub K \cap G^{ad}(\A_f^{p})$. Then for any open compact subgroup $K_p \sub G(\Qp)$, there is a natural finite morphism $\mathcal{X}_{K^p K_p}^{\ast}(G,X)^0 \to \mathcal{X}_{\Gamma \cap \pi(K_p)}^{\ast}(G^{ad},X^+)$. Moreover, these morphisms are compatible as $K_p$ varies, and the transition maps in the two towers are finite. Passing to the inverse limit over $K_p$, the result now follows from Lemma \ref{twotowers}.
\end{proof}



We now come to the key result in this subsection.

\begin{prop}\label{going down}
Let $(G,X)$ be a Shimura datum of Hodge type. Then the connected Shimura datum $(G^{ad}, X^+)$ satisfies Property $\mathcal{P}$.
\end{prop}

\begin{proof}  We start by proving that 
\[ 
\varprojlim_{K_p \subset G^{ad}(\mathbb{Q}_p)} \mathcal{X}^{\ast}_{\Gamma \cap K_p}(G^{ad},X^+)^{\lozenge}
\] 
is a perfectoid space when $\G \sub G^{ad}(\Q)^+$ is a congruence subgroup. Let $\pi: G \to G^{ad}$ denote the natural map. Choose a congruence subgroup $\Gamma' = K \cap G(\Q)_+ \sub G(\Q)_{+}$ with $\pi(\Gamma') \sub \Gamma$, and set $\Gamma'' = \Gamma' \cap G^{der}(\Q)$, so $\Gamma''$ is also a congruence subgroup. Choose a cofinal descending family of open compact subgroups 
\[
K_{p,0} \supseteq K_{p,1} \supseteq \cdots \supseteq K_{p,n} \supseteq \cdots
\]
in $G(\Qp)$, and write $K_{p,n}^{der} = K_{p,n} \cap G^{der}(\Qp)$. Without loss of generality, we can assume that $K_{p,0}^{der} \cap Z_G(\Qp) = \{ 1 \}$ and that $\Gamma' \sub K_{p,0}$, so then $\Gamma'' \sub K_{p,0}^{der}$ and $ \Gamma'' \cap Z_G(\Qp) = \{ 1 \}$, and the map $\pi$ induces isomorphisms $\pi(\Gamma'' \cap K_{p,n}) = \pi(\Gamma'' \cap K_{p,n}^{der}) = \pi(\Gamma'') \cap \pi(K_{p,n}^{der})$. Moreover, the inclusion $\Gamma'' \sub \Gamma'$ induces a natural map of towers 
\[ 
(\mathcal{X}^{\ast}_{\pi(\Gamma'' \cap K_{p,n})}(G^{ad},X^+))_{n \geq 0} \to (\mathcal{X}^{\ast}_{K \cap K_{p,n}}(G,X))_{n \geq 0} 
\] 
where the map at every level $n$ is finite. By Corollary \ref{hodgetypeperfectoidgen}, the target of this map is a good tower.

\medskip

Now define $\Gamma''' = \cap_{\gamma \in \Gamma / \pi(\Gamma'')} \gamma \pi(\Gamma'') \gamma^{-1}$. By design, $\Gamma'''$ is an arithmetic subgroup of $G^{ad}(\Q)^+$, and is a normal subgroup of $\Gamma$ with finite index. Since $\Gamma''' \cap \pi(K_{p,n}^{der})$ is of finite index in $\pi(\Gamma'') \cap \pi(K_{p,n}^{der}) = \pi(\Gamma'' \cap K_{p,n})$, we get another natural map of towers 
\[ 
(\mathcal{X}^{\ast}_{\Gamma''' \cap \pi(K_{p,n}^{der})}(G^{ad},X^+))_{n \geq 0} \to (\mathcal{X}^{\ast}_{\pi(\Gamma'' \cap K_{p,n})}(G^{ad},X^+))_{n \geq 0} 
\]
where the map at every level $n$ is finite. For any $n\geq 0$, $\Gamma''' \cap \pi(K_{p,n}^{der})$ is a normal finite-index subgroup of $\Gamma \cap \pi(K_{p,n}^{der})$. Set $\Delta_n = (\Gamma''' \cap \pi(K_{p,n}^{der}) )\backslash (\Gamma \cap \pi(K_{p,n}^{der}))$, so $\Delta_n$ is a finite group and the natural maps $\Delta_{n+1} \to \Delta_{n}$ are injective. Write $\Delta = \varprojlim_{n} \Delta_n$, so $\Delta = \Delta_n$ for all sufficiently large $n$. Then $\Delta$ operates naturally on the tower $(\mathcal{X}^{\ast}_{\pi(\Gamma''') \cap \pi(K_{p,n}^{der})}(G^{ad},X^+))_{n \geq 0}$, and $\mathcal{X}^{\ast}_{\Gamma''' \cap \pi(K_{p,n}^{der})}(G^{ad},X^+)/\Delta \cong \mathcal{X}^{\ast}_{\Gamma \cap \pi(K_{p,n}^{der})}(G^{ad},X^+)$ for all sufficiently large $n$.

\medskip

Summarizing the situation so far, we have a diagram of towers
\[ \xymatrix{(\mathcal{X}^{\ast}_{\pi(\Gamma''') \cap \pi(K_{p,n}^{der})}(G^{ad},X^+))_{n \geq 0} \ar[r] \ar[d]& (\mathcal{X}^{\ast}_{\pi(\Gamma'' \cap K_{p,n})}(G^{ad},X^+))_{n \geq 0} \ar[r] &  (\mathcal{X}^{\ast}_{K \cap K_{p,n}}(G,X))_{n \geq 0} \\
(\mathcal{X}^{\ast}_{\Gamma \cap \pi(K_{p,n}^{der})}(G^{ad},X^+))_{n \geq 0}
}
\]
where all the morphisms at any given level $n$ are finite. We've already observed that the upper-right tower is a good tower, so by two applications of Proposition \ref{goodtowerproperties}(i), we deduce that the upper-left tower is a good tower. Since $\Delta$ operates naturally on the upper-left tower and $\mathcal{X}^{\ast}_{\Gamma''' \cap \pi(K_{p,n}^{der})}(G^{ad},X^+)/\Delta \cong \mathcal{X}^{\ast}_{\Gamma \cap \pi(K_{p,n}^{der})}(G^{ad},X^+)$ for all sufficiently large $n$, we may apply Proposition \ref{goodtowerproperties}(ii) to deduce that $\mathcal{X}^{\ast}_{\Gamma''',\infty}(G^{ad},X^+) / \Delta$ is a perfectoid space and that $\mathcal{X}^{\ast}_{\Gamma''',\infty}(G^{ad},X^+) / \Delta \cong  \varprojlim_n \mathcal{X}^{\ast}_{\Gamma''' \cap \pi(K_{p,n}^{der})}(G^{ad},X^+)/\Delta$. But  $\varprojlim_n \mathcal{X}^{\ast}_{\Gamma''' \cap \pi(K_{p,n}^{der})}(G^{ad},X^+)/\Delta \cong \varprojlim_n \mathcal{X}^{\ast}_{\Gamma \cap \pi(K_{p,n}^{der})}(G^{ad},X^+) = \mathcal{X}^{\ast}_{\Gamma,\infty}(G^{ad},X^+)$, so we conclude that $ \mathcal{X}^{\ast}_{\Gamma,\infty}(G^{ad},X^+)$ is a perfectoid space, as desired. This finishes the proof when $\G$ is a congruence subgroup.

\medskip

Now assume that $\G \sub G^{ad}(\Q)^+$ is an arithmetic subgroup. By Propositions \ref{arithmetic contained in congruence} and \ref{deligne}, there is a congruence subgroup $\G^\prime$ such that $\G \sub \G^\prime \sub G^{ad}(\Q)^+$. Then 
\[
\left( \mc{X}^{\ast}_{\G \cap K_p}(G^{ad},X^+) \right)_{K_p \sub G^{ad}(\Qp)} \to \left( \mc{X}^{\ast}_{\G^\prime \cap K_p}(G^{ad},X^+) \right)_{K_p \sub G^{ad}(\Qp)}
\]
is map of towers with finite transition maps, and by above $\mc{X}_{\G^\prime,\infty}(G^{ad},X^+)$ is perfectoid. By Lemma \ref{twotowers} $\mc{X}_{\G,\infty}(G^{ad},X^+)$ is perfectoid, as desired.
\end{proof}

We may now summarize our results in this section in the following theorem.

\begin{theo}\label{main p-adic}
Let $(G,X)$ be a Shimura datum (resp. a connected Shimura datum) of pre-abelian type. Then, for any compact open subgroup $K^p \sub G(\A_f)$ (resp. arithmetic subgroup $\G \sub G^{ad}(\Q)^+$), the diamond $\mc{X}_{K^p}^\ast(G,X)$ (resp. $\mc{X}_{\G,\infty}^\ast(G,X)$) is a perfectoid space.
\end{theo}

\begin{proof}
Choose a Shimura datum $(G_1,X_1)$ of Hodge type with a central isogeny $G_1^{der} \to G^{ad}$ inducing an isomorphism $(G_1^{ad},X_1^+) \cong (G^{ad},X^+)$. By Proposition \ref{going down}, $(G^{ad},X^+)$ satisfies property $\mc{P}$, and then Proposition \ref{going up} implies that $(G,X)$ satisfies property $\mc{P}$, as desired.
\end{proof}

This has the following consequence for compactly supported completed cohomology, which may be viewed as a generalization of \cite[Corollary 4.2.2]{scho}.

\begin{coro}\label{compactly supported vanishing}
Let $(G,X)$ be a connected Shimura datum of pre-abelian type. Then Conjecture \ref{connected vanishing} for $?=c$ holds for $G$.
\end{coro}

\begin{proof}
Note that the towers used to formulate Conjecture \ref{connected vanishing} correspond to the towers used in this section. Once we know that the towers of minimal compactifications are perfectoid in the limit (by Theorem \ref{main p-adic}), the argument in the proof of \cite[Corollary 4.2.2]{scho} (and \cite[Theorem 4.2.1]{scho}, on which it relies) goes through verbatim. Note that the boundary is strongly Zariski closed (which is need for the argument proving  \cite[Theorem 4.2.1]{scho}), since all Zariski closed sets are strongly Zariski closed by \cite[Remark 7.5]{bhatt-scholze}.
\end{proof}

\begin{rema}\label{comments on cs}
We give some remarks on the possibility of proving vanishing above the middle degree for $\wt{H}^\ast$ using perfectoid methods instead of the topological methods used in sections \ref{sec: cc} and \ref{sec: shimura}. In \cite{caraiani-scholze2}, Caraiani and Scholze prove that toroidal compactifications of certain unitary Shimura varieties are perfectoid in the limit and that the (\'etale) cohomology of this perfectoid space computes completed cohomology, which implies the desired vanishing (see \cite[Theorem 2.6.2, Lemma 4.6.2]{caraiani-scholze2}). The perfectoidness result relies on a result of Pilloni--Stroh \cite{pilloni-stroh} for Siegel modular varieties. It seems to us that these methods should extend directly to Shimura varieties of Hodge type. However, the more general case of abelian type is not clear to us. In particular, in the general abelian type case, it is not clear to us that the auxiliary cone decompositions can be chosen so that the toroidal tower becomes perfectoid at infinite level, while simultaneously ensuring that the transition maps over the boundary have the good behavior required for the argument in \cite{caraiani-scholze2}.

\medskip

We also note that the perfectoid methods do not directly give that $\wt{H}^i_c \to \wt{H}^i$ is an isomorphism or surjective in a range of degrees including the middle. In principle, however, there is a connection between the perfectoid method and the method of this paper. The result \cite[Lemma 4.6.2]{caraiani-scholze2}, which essentially goes back to Pink \cite{pink}, morally says that infinite level toroidal compactifications behave like Borel--Serre compactifications. Thus, one could get more detailed information from the perfectoid method by studying the map from the toroidal compactification to the minimal compactification, as in \cite{pink}. Morally, this should give the same information in the end as the topological method in this paper. However, in our opinion, our topological method is far more elementary and transparent, and far less technically demanding. 
\end{rema}

\bibliographystyle{alpha}
\bibliography{abelianbib}

\end{document}